\documentclass[10pt]{amsart}           

\usepackage{amssymb}
\usepackage{amsmath}
\usepackage{amsfonts}
\usepackage{amsthm}
\usepackage{mathrsfs}
\usepackage{color}

\renewcommand\mathcal{\mathscr}

\theoremstyle{plain}
\newtheorem{theorem}{Theorem}[section]
\newtheorem*{theorem*}{Theorem}
\newtheorem{lemma}[theorem]{Lemma}

\newtheorem{proposition}[theorem]{Proposition}
\theoremstyle{remark}
\newtheorem{remark}[theorem]{Remark}
\numberwithin{equation}{section}
\theoremstyle{definition}
\newtheorem{definition}[theorem]{Definition}
\numberwithin{equation}{section}
\theoremstyle{notation}

\numberwithin{equation}{section}
\theoremstyle{Basic assumptions}
\newtheorem{Basic assumptions}[theorem]{Basic assumptions}
\numberwithin{equation}{section}

\newcount\quantno
\everydisplay{\quantno=0}\everycr{\quantno=0}
\newcommand\quant{\advance\quantno by1
                      \ifnum\quantno=1\qquad\else\quad\fi\forall }
\newcommand\itemno[1]{(\romannumeral #1)}

\renewcommand\Re{\operatorname{\mathrm{Re}}}
\renewcommand\Im{\operatorname{\mathrm{Im}}}

\newcommand\rest[1]{\kern-.1em
          \lower.5ex\hbox{$\scriptstyle #1$}\kern.05em}

\newcommand\diam[1]{\mathrm{diam}#1}
\newcommand\set[1]{{\left\{#1\right\}}}

\newcommand\bigset[1]{\bigl\{#1\bigr\}}

\renewcommand\mod[1]{\vert{#1}\vert}
\newcommand\bigmod[1]{\bigl\vert{#1}\bigr|}
\newcommand\Bigmod[1]{\Bigl\vert{#1}\Bigr|}

\newcommand\norm[2]{{\Vert{#1}\Vert_{#2}}}
\newcommand\normto[3]{{\Vert{#1}\Vert_{#2}^{#3}}}
\newcommand\bignorm[2]{\big\Vert{#1}\big\Vert_{#2}}
\newcommand\bignormto[3]{\big\Vert{#1}\big\Vert_{#2}^{#3}}

\newcommand\opnorm[2]{|\!|\!| {#1} |\!|\!|_{#2}}
\newcommand\bigopnorm[2]{\bigl|\!\bigl|\!\bigl| {#1} 
\bigr|\!\bigr|\!\bigr|_{#2}}

\newcommand\prodo[2]{\left\langle#1,#2\right\rangle}

\newcommand\wrt{\,\text{\rm d}}

\newcommand\bS{\mathbf{S}}

\newcommand\BC{\mathbb{C}}
\newcommand\BD{\mathbb{D}}

\newcommand\BN{\mathbb{N}}
\newcommand\BR{\mathbb{R}}

\newcommand\BZ{\mathbb{Z}}

\newcommand\cB{\mathcal{B}}   
\newcommand\frb{\mathfrak{b}}

\newcommand\cD{\mathcal{D}}

\newcommand\cF{\mathcal{F}}

\newcommand\cH{\mathcal{H}}   
\newcommand\frh{\mathfrak{h}} 
\newcommand\cI{\mathcal{I}}

\newcommand\cL{\mathcal{L}}    
\newcommand\cM{\mathcal{M}}  \newcommand\fM{\mathfrak{M}} \newcommand\frm{\mathfrak{m}}  
    
   \newcommand\fro{\mathfrak{o}} 
     
\newcommand\cQ{\mathcal{Q}}
\newcommand\cR{\mathcal{R}}    
\newcommand\cS{\mathcal{S}}    
\newcommand\cT{\mathcal{T}}

\newcommand\al{\alpha}
\newcommand\be{\beta}
\newcommand\ga{\gamma}    
\newcommand\de{\delta}
\newcommand\ep{\epsilon}  \newcommand\vep{\varepsilon}

\newcommand\la{\lambda}   
\newcommand\om{\omega}      

\newcommand\te{\theta}
\newcommand\vp{\varphi}

\newcommand\OV{\overline}
\newcommand\funnyk{k\hbox to 0pt{\hss\phantom{g}}}

\newcommand\lu[1]{L^1(#1)}
\newcommand\lp[1]{L^p(#1)}

\newcommand\ld[1]{L^2(#1)}

\newcommand\ly[1]{L^\infty(#1)}

\newcommand\hu[1]{H^1(#1)}
\newcommand\ghu[1]{{\frak h}^1(#1)}
\newcommand\ghuI[1]{{\frak h}_I^1(#1)}

\newcommand\gbmo[1]{\frb\frm\fro(#1)}

\newcommand\wh{\widehat}
\newcommand\wt{\widetilde}
\newcommand\whH{\widehat{\phantom{G}}\hbox to 0pt{\hss $H$}}

\newcommand\emspace{\hbox to 6pt{\hss}}
\newcommand\ds{\displaystyle}

\newcommand\rmi{\hbox{\rm (i)}}
\newcommand\rmii{\hbox{\rm (ii)}}
\newcommand\rmiii{\hbox{\rm (iii)}}
\newcommand\rmiv{\hbox{\rm (iv)}}
\newcommand\rmv{\hbox{\rm (v)}}

\newcommand\ioty{\int_0^{\infty}}

\newcommand\One{{\mathbf{1}}}
\newcommand\OU{Ornstein--Uhlenbeck}

\newcommand\e{\mathrm{e}}

\DeclareSymbolFont{EUEX}{U}{euex}{m}{n}

\DeclareSymbolFont{euexlargesymbols}{U}{euex}{m}{n}
\DeclareMathSymbol{\intop}{\mathop}{euexlargesymbols}{"52}
     \def\int{\intop\nolimits}

\DeclareSymbolFont{euexsymbols}     {U}{euex}{m}{n}
\DeclareMathSymbol{\smallint}{\mathop}{euexsymbols}{"52}

\begin{document}

\title[Spaces of Goldberg type]
{Spaces of Goldberg type \\
\vskip0.3cm
on certain measured metric spaces}

\subjclass[2000]{} 


\thanks{Work partially supported by PRIN 2010 ``Real and complex manifolds: 
geometry, topology and harmonic analysis". 
}

\author[S. Meda and S. Volpi]
{Stefano Meda and Sara Volpi}


\address{Stefano Meda: 
Dipartimento di Matematica e Applicazioni
\\ Universit\`a di Milano-Bicocca\\
via R.~Cozzi 53\\ I-20125 Milano\\ Italy
\hfill\break 
stefano.meda@unimib.it}

\address{Sara Volpi: 
Dipartimento di Matematica e Applicazioni
\\ Universit\`a di Milano-Bicocca\\
via R.~Cozzi 53\\ I-20125 Milano\\ Italy
\hfill\break 
s.volpi@yahoo.it}



\maketitle

\begin{abstract}
In this paper we define a space $\ghu{M}$ of Hardy--Goldberg type on a measured
metric space satisfying some mild conditions.  We prove that the dual of $\ghu{M}$
may be identified with $\gbmo{M}$, a space of functions with ``local''
bounded mean oscillation, and that if $p$ is in $(1,2)$, then $\lp{M}$ is
a complex interpolation space between $\ghu{M}$ and $\ld{M}$.
This extends previous results
of Strichartz, Carbonaro, Mauceri and Meda, and Taylor.  
Applications to singular integral operators on Riemannian manifolds are given.   
\end{abstract}

\setcounter{section}{0}
\section{Introduction}\label{s: Introd}

This paper focuses on the study of spaces of Hardy--Goldberg type on certain 
measured metric spaces.  Our goal is twofold: on the one hand we aim at
extending previous work on the subject by R.S. Strichartz \cite{Str}, 
A.~Carbonaro, G.~Mauceri and Meda \cite{CMM1,CMM2}, 
and M.~Taylor \cite{T2}.  On the other hand, our results pave the
way to further developments concerning Riesz
transforms on a certain class of noncompact Riemannian manifolds, 
that will appear in a forthcoming paper \cite{CMV}.   

Strichartz worked on compact Lie groups; some of his far reaching ideas 
have been subsequently developed by Taylor to successfully extend Strichartz's results 
to the setting of Riemannian manifolds with strongly bounded geometry. 
A comparison between the results contained in \cite{CMM1} and \cite{T2} may help understanding
our motivations and contributions.

In \cite{CMM1} the authors consider a metric measured space $(M,\mu,d)$
satisfying three conditions: the
approximate midpoint property (AMP), the local doubling condition (LDC) and Cheeger's
isoperimetric property (IP) (see Section~\ref{sec:1} for the definitions).  
The AMP is a very mild assumption, very often satisfied,
the LDC is a very natural assumption for the applications we have in mind to 
Riemannian manifolds, whereas
the IP is a comparatively restrictive assumption, for it implies that the volume growth
of $M$ be at least exponential \cite[Proposition~3.1~\rmi]{MMV1}.  
In this setting the authors introduce an atomic Hardy space $\hu{M}$,   
identify the dual space of $\hu{M}$ with $BMO(M)$ (suitably defined), 
and prove that if $p$ is in $(2,\infty)$, then $\lp{M}$ is an
interpolation space between $\ld{M}$ and $BMO(M)$.  
Also, applications to spectral multipliers and Riesz transforms are given.  
It is important to keep in mind that 
atoms are functions in $\ld{M}$, with support contained in balls of radius at most $1$, say, 
satisfying the standard size estimate and cancellation property (the same 
as those satisfied by atoms in the classical Hardy space $\hu{\BR^n}$).  

In \cite{T2}, Taylor works on a Riemannian manifold $M$ of bounded geometry in a
very strong sense, which requires a uniform local control of all derivatives of the metric tensor
in exponential co-ordinates around each point, but 
a mild control on the volume growth of the manifold.  He defines a 
local Hardy space $\ghu{M}$, which is a direct generalisation of the classical
local Hardy space $\ghu{\BR^n}$, introduced by Goldberg \cite{G}, and of the extension thereof
to compact Lie groups by Strichartz.  Taylor defines $\ghu{M}$
via a suitable grand maximal function, identifies the dual space of $\ghu{M}$ with 
$\gbmo{M}$ (suitably defined), and proves that if $p$ is in $(2,\infty)$, then $\lp{M}$ is an
interpolation space between $\ld{M}$ and $\gbmo{M}$.  Applications to a wide class
of pseudo-differential operators are provided.  
Taylor also proves that $\ghu{M}$ has an atomic decomposition, whose atoms are either 
atoms in $\hu{M}$ (in the sense of \cite{CMM1}), or functions in $\ld{M}$, 
supported in a ball of radius exactly equal to $1$ and satisfying the standard 
size condition, but possibly not the cancellation condition.  
One of the limitations of this approach is that the geometric assumptions on the 
Riemannian manifolds, are, as mentioned above, quite stringent.  One of its advantages
is that it reduces any estimate involving $\ghu{M}$ to corresponding local estimates for $\ghu{\BR^n}$.

It is worth observing that each of the spaces $\hu{M}$ and $\ghu{M}$ has its own
adavantages and range of applications.  Clearly $\ghu{M}$ is a flexible space that is preserved
by the action of suitable classes of pseudo-differential operators \cite{G,Str,T2}.  However,
it is not apt to obtaining endpoint estimates for certain singular integral operators 
like, for instance, the purely imaginary powers of the translated \OU\ operator 
\cite{CMM3}, where $\hu{M}$ functions perfectly.  

As mentioned above, one of the motivation of our work is 
to extend considerably the range of applicability
of the approach of Strichartz and Taylor.  
Our ambient space is a measured metric space possessing AMP and LDP.
It is well known that the assumptions
above are satisfied whenever $M$ is a Riemannian manifold with Ricci curvature bounded
from below (without assuming that $M$ has positive injectivity radius),
a condition that does not require any control on the derivatives of the metric tensor.  
Note that such manifolds may have exponential volume growth, so that they may
not be homogeneous spaces in the sense of Coifman--Weiss. 
Note that we do not assume that $M$ possesses the so called uniform ball size condition, i.e., 
it may happen that $\inf \, \{\mu(B): r_B=r\} = 0$
and $\sup \, \{\mu(B): r_B=r\} = +\infty$ for each $r>0$.  

We emphasize the fact that our methods are quite different
from those of Taylor, for we cannot reduce the analysis to 
that of Goldberg on Euclidean spaces.  
We give an atomic definition of $\ghu{M}$: when $M$ is a manifold of strongly bounded
geometry, $\ghu{M}$ agrees with the space defined by Taylor.  
%
We prove that the topological dual of $\ghu{M}$ may be identified with 
a local space $\frb\frm\fro (M)$ of functions of 
bounded mean oscillation in an appropriate sense (see Sections~\ref{sec:4} and
\ref{sec:5}), 
and that if $p \in (1,2)$, then $\lp{M}$ is
a complex interpolation space between $\ghu{M}$ and $\ld{M}$ (see Section~\ref{sec:7}).
Applications to the study
of the translated Riesz transform and of spectral multipliers of the Laplace--Beltrami
operator on manifolds with Ricci curvature bounded from below 
will be given in Section~\ref{sec:9}.

Finally, a few words concerning our second goal.
A basic question concerning the Riesz transform $\cR = \nabla\cL^{-1/2}$ (here $\cL$ denotes the 
Laplace--Beltrami operator on $M$) is to characterise the space $H^1_{\cR}(M)$ of all functions 
$f$ in $\lu{M}$ such that $\bigmod{\cR f}$ is in $\lu{M}$.  
In many cases, for instance in $\BR^n$, such space is just the Hardy space $\hu{\BR^n}$.
Recent results of Mauceri, Meda and M.~Vallarino \cite{MMV4} show that if $\BD$
denotes the hyperbolic disc, then $H^1_{\cR}(\BD)$ is \emph{not} $\hu{\BD}$.   
The analysis of $\ghu{M}$ made in this paper will be the key to provide 
a characterisation of $H_\cR^1(M)$ for a comparatively large class of Riemannian manifolds.  

We will use the ``variable constant convention'', and denote by $C$,
possibly with sub- or superscripts, a constant that may vary from place to
place and may depend on any factor quantified (implicitly or explicitly)
before its occurrence, but not on factors quantified afterwards.

For each $p$ in $[1,\infty]$, we denote by $p'$ the index conjugate to $p$, i.e. $p' =  p/(p-1)$.

\section{Notation, terminology and geometric assumptions}\label{sec:1}

Suppose that $(M,d,\mu)$ is a measured metric space,
and denote by $\cB$ the family of all balls on $M$.
We assume that $\mu(M) > 0$ and that every ball has finite measure.
For each $B$ in $\cB$ we denote by $c_B$ and $r_B$
the centre and the radius of $B$ respectively.  
Furthermore, we denote by $k B$ the ball with centre $c_B$ and radius $k r_B$.  For each $s$ in $\BR^+$, 
we denote by $\cB_s$ the family of all balls $B$ in $\cB$ such that $r_B \leq s$.

\medskip
We say that $M$ possesses the \emph{local doubling property} (LDP)
if for every $s$ in $\BR^+$ there exists a constant $D_s$ 
such that
\begin{equation*}  \label{f: LDC} 
\mu \bigl(2 B\bigr)
\leq D_s \, \mu  \bigl(B\bigr)
\quant B \in \cB_s.
\end{equation*}

\begin{remark} \label{r: geom I}
The LDP implies that for each $\tau \geq 1$
and for each $s$ in $\BR^+$
there exists a constant~$C$ such that
\begin{equation} \label{f: doubling Dtau}
\mu\bigl(B'\bigr)
\leq C \, \mu(B)
\end{equation}
for each pair of balls $B$ and $B'$, with $B\subset B'$,
$B$ in $\cB_s$, and $r_{B'}\leq \tau \, r_B$.
We shall denote by $D_{\tau,s}$ the smallest constant for
which (\ref{f: doubling Dtau}) holds.
In particular, if (\ref{f: doubling Dtau}) holds (with the same constant)
for all balls $B$ in $\cB$, then $\mu$ is doubling and
we shall denote by $D_{\tau,\infty}$ the smallest constant for
which (\ref{f: doubling Dtau}) holds.
\end{remark}

We say that $M$ possesses the \emph{approximate midpoint property} (AMP) if
there exist $R_0$ in $[0,\infty)$ and 
$\be$ in $[1/2,1)$ such that for every pair of points~$x$ and $y$
in $M$ with $d(x,y) > R_0$ there exists a point $z$ in $M$ such that
$d(x,z) < \be\, d (x,y)$ and $d(y,z) < \be\, d (x,y)$.
This is clearly equivalent to the requirement that there exists a
ball~$B$ containing~$x$ and $y$ such that $r_B < \be \, d(x,y)$.

If $M$ is a measured metric space for which $R_0 =0$ and each segment has a midpoint,
then we say that $M$ possesses the midpoint property (MP).
Typically graphs enjoy the AMP, but quite often a ``segment'' in a graph has not a midpoint.
On the other hand, every connected Riemannian manifold possesses the~MP, and the constant
$R_0$ is equal to $0$.   

All the results in this paper hold under the assumption that $M$ possesses
the local doubling property LDP and the approximate midpoint property AMP.  However,
for the sake of simplicity, \textbf{hereafter we assume that $M$ possesses
the local doubling property LDP and the midpoint property MP} (with $R_0 = 0$).  This leads
to cleaner statements, and allows us to avoid certain annoying technicalities, which
makes the reading more difficult.  The interested reader may easily fill the additional details
and come to prove our results under the assumption that $M$ satisfies the AMP only.  
To this end, \cite{CMM1} may serve as a guide, for the details of proofs therein are
done under the assumption that $M$ possesses the AMP only.  

Given a positive number $\eta$,
a set $\fM$ of points in $M$ is a \emph{$\eta$-discretisation} of $M$
if it is maximal with respect to the following property:
$$
\min\{d(z,w): z,w \in \fM, z \neq w \} >\eta\quad
\hbox{and}\quad d(\fM, x) \leq \eta 
\quant x \in M.  
$$
It is straightforward to show that $\eta$-discretisations exist
for every $\eta$.  For each subset $E$ of $M$, we set 
$$
\fM_E
:=   \big\{z \in \fM: B_{2\eta}(z) \cap E \neq \emptyset \big\},
$$
and denote by $\sharp \fM_E$ its cardinality.
If $x$ is a point in $M$, we write $\fM_x$ instead of $\fM_{\{x\}}$, for simplicity.  
Note that $\sharp\fM_x$ is the number of balls of the covering $\set{B_{2\eta}(z): z \in \fM }$
that contain $x$.  

\begin{lemma}\label{l:bolle}
Suppose that $M$ possesses the LDP and the MP (with $R_0=0$, see, however, the remark before the definition
of discretisations).
Assume that $c$ is a positive number and that $\fM$ is a $c/2$-discretisation.
The following hold:
\begin{itemize}
\item[\itemno1] 
the family $\set{B_c(z): z\in \fM}$ is a locally uniformly finite
covering of $M$, and there exists a constant $C$, depending on $c$, such that 
$
\sup_{x \in M} \, \sharp\fM_x 
\leq C;
$
\item[\itemno2] 
for every $b > c$ there exists a constant $C$, which depends on $b$ and $c$,
such that 
$
\sharp \fM_B 
\leq C  
$
for every ball $B$ of radius $b$.
\end{itemize}
\end{lemma}

\begin{proof}
First we prove \rmi. 
Since $\fM$ is a $c/2$-discretisation, $d(\fM,x)\leq c/2$ for every $x$ in $M$, so that 
$\set{B_c(z): z \in \fM}$ is a covering of $M$.  
Observe that if $z$ is in $\fM_x$, then 
$
B_c(z) \subset B_{2c}(x) \subset B_{3c}(z).
$ 
This and the LDP \eqref{f: doubling Dtau}  imply that 
$$
\mu\big(B_{2c}(x)\big) \leq \mu\big(B_{3c}(z)\big) \leq D_{12,c/4}\, \mu\big(B_{c/4}(z)\big).
$$
Since $B_{c/4}(z) \subset B_{2c}(x)$ and the balls of the family $\set{B_{c/4}(z): 
z \in \fM}$ are pairwise disjoint, 
\begin{align*}
\mu\big(B_{2c}(x)\big) 
& \geq \mu \big(\bigcup_{z\in \fM_x} B_{c/4}(z)\big)  \\ 
& =    \sum_{z\in \fM_x}\mu\big(B_{c/4}(z)\big)  \\
& \geq \frac{\sharp\fM_x}{D_{12,c/4}}\,\mu\big(B_{2c}(x)\big),  
\end{align*}
whence $\sharp\fM_x \leq D_{12,c/4}$, as required. 

\medskip
Now we prove \rmii. 
Denote by $B'$ the ball with centre $c_B$ and radius $b+2c$.  Observe that 
if $z$ is in $\fM_B$, and $x$ belongs to $B_c(z)$, then $d(x, c_B) < b+2c$.  
Therefore $x$ is in $B'$.   This and \rmi\ imply that 
$$
\sum_{z\in \fM_B} \One_{B_c(z)}
\leq D_{12,c/4}\, \One_{B'}. 
$$
By integrating both sides of this inequality, we see that
$$
\sum_{z\in \fM_B} \mu\big(B_{c/4}(z) \big)
\leq \sum_{z\in \fM_B} \mu\big(B_c(z) \big)\leq D_{12,c/4}\, \mu(B').
$$
Recall that the balls $B_{c/4}(z)$, $z \in \fM_B$, are
pairwise disjoint, and that $\mu\big(B_{c/4}(z)\big) \geq D_{4(b/c)+8}^{-1} \, \mu(B')$
by the LDP, so that 
$$
\sharp\fM_B \,D_{4(b/c)+8}^{-1} \, \mu(B') 
\leq D_{12,c/4}\, \mu(B'),
$$
from which the required estimate follows directly.  
\end{proof}

\begin{remark} \label{r: constants}
A careful examination of the proof of Lemma~\ref{l:bolle} reveals that, in fact, 
we have proved the following: $\sup_{x\in M} \, \sharp\fM_x \leq D_{12,c/4}$ and 
$\sharp \fM_B \leq D_{4(b/c)+8}D_{12,c/4}$ (see Remark~\ref{r: geom I} for the definition of $D_{\tau,s}$).
We have made here the choice not to keep track of the precise dependence of the constants appearing
in the statement from the various parameters.  We shall do the same in all the subsequent sections. 
\end{remark}

\section{The local Hardy space $\frh^1(M)$} \label{sec:2}

\begin{definition} \label{d: atom}
Suppose that $p$ is in $(1,\infty]$ and let $p'$ denote the index conjugate to $p$. 
Suppose that $b$ is a positive number.
A \emph{standard $p$-atom at scale $b$}
is a function $a$ in $L^1(M)$ supported in a ball $B$ in $\cB_b$
satisfying the following conditions:
\begin{enumerate}
\item[\itemno1] \emph{size condition}:\\
$\norm{a}{\infty}\leq \mu (B)^{-1}$ if $p=\infty$ and 
$\norm{a}{p}  \leq \mu (B)^{-1/p'}$ if $p\in(1, \infty)$;
\item[\itemno2] \emph{cancellation condition}:\\
$\ds \int_B a \wrt \mu  = 0$. 
\end{enumerate}
A \emph{global $p$-atom at scale $b$} is a function $a$
in $L^1(M)$ supported in a ball $B$ of radius \emph{exactly equal to} $b$ 
satisfying the size condition above (but possibly not the cancellation condition).
Standard and global $p$-atoms will be referred to simply as $p$-\emph{atoms}.
\end{definition}

\begin{definition} \label{d: Goldberg}
Let $b$ be a positive number. 
The \emph{local atomic Hardy space} $\frh^{1,p}_b({M})$ is the 
space of all functions~$f$ in $L^1(M)$
that admit a decomposition of the form
\begin{equation} \label{f: decomposition}
f = \sum_{j=1}^\infty \la_j \, a_j,
\end{equation}
where the $a_j$'s are $p$-atoms at scale $b$
and $\sum_{j=1}^\infty \mod{\la_j} < \infty$.
The norm $\norm{f}{\frh^{1,p}_b}$
of $f$ is the infimum of $\sum_{j=1}^\infty \mod{\la_j}$
over all decompositions (\ref{f: decomposition}) of $f$.
\end{definition}

We shall prove that $\frh^{1,p}_b(M)$ is independent of $p$ and $b$, and later the space
$\frh^{1,p}_1(M)$ will be denoted simply by $\ghu{M}$.

The following lemma produces an economical decomposition
of atoms supported in ``big'' balls as finite linear combinations
of atoms supported in smaller balls. 
This result extends to global atoms the economical decomposition for 
standard atoms proved in \cite[Lemma~6.1]{MMV3}; see also \cite[Prop~4.3 (i)]{CMM1}
for a ``less economical'' decomposition. 
It is worth observing that our proof does not require the uniform ball size condition, 
which, instead, is used in \cite[Lemma~6.1]{MMV3}.  
Furthermore the proof of the following lemma is somewhat simpler than the proof of \cite[Lemma~6.1]{MMV3},
for we can decompose atoms supported in ``big'' balls as finite linear combinations of
\emph{global} atoms supported in smaller balls, so that we need not care about cancellations.  

\begin{lemma} \label{l: economical decomposition}
Suppose that $p$ is in $(1,\infty]$ and that $b>c>0$. 
Then each $p$-atom $a$ at scale $b$ may be written as a finite linear
combination of global $p$-atoms at scale~$c$, and there exists a constant $C$, independent of
the atom $a$, such that 
$
\norm{a}{\frh^{1,p}_c} 
\leq C.  
$
\end{lemma}

\begin{proof}
Suppose that $a$ is a $p$-atom at scale~$b$ (either standard or global), supported
in the ball $B$, and denote by $\fM$ a $c/2$-discretisation of $M$. 
We denote by $B_1,\ldots, B_N$ the balls 
with centre at points in $\fM_B$ and radius $c$, and define
$$
\psi_j 
:= \frac{\One_{B_j}}{\sum_{k=1}^{N} \One_{B_k}}.
$$
Clearly $\sum_{k=1}^N \, \psi_j$ is equal to $1$ on $B$.  
Set 
$\la_j:=\norm{a\psi_j}{p}\, \mu (B_j)^{1/p'}$,  $b_j := a\psi_j \, \la_j^{-1}$,  
and write
$
a 
= \sum_{j=1}^N \psi_j \, a
= \sum_{j=1}^N  \la_j\, b_j. 
$
Clearly $b_j$ is a global $p$-atom at scale $c$, whence
$$
\begin{aligned}
\norm{a}{\frh^{1,p}_c} 
& \leq \sum_{j=1}^N  \mod{\la_j} \\
& =    \sum_{j=1}^N \norm{a\psi_j}{p}\,   \mu (B_j)^{1/p'}  \\ 
& \leq \Big[\sum_{j=1}^N \normto{a\psi_j}{p}{p}\Big]^{1/p} \,   
       \Big[\sum_{j=1}^N \mu (B_j)  \Big]^{1/p'};
\end{aligned}
$$
we have used H\"older's inequality with exponents $p$ and $p'$ in the last
inequality.  Observe that the balls $B_j$ are contained in the ball with
centre $c_B$ and radius $b+2c$.  Since, by Lemma~\ref{l:bolle}~\rmi\,
each point in $B_{b+2c}(c_B)$ is covered by at most $C$ balls $B_j$, with $C$
depending only on~$c$,  
$$
\sum_{j=1}^N \mu (B_j) 
\leq C \,  \mu\big(B_{b+2c}(c_B)\big).  
$$ 
Similarly, 
$$
\begin{aligned}
\sum_{j=1}^N \normto{a\psi_j}{p}{p}
& = \int_M \sum_{j=1}^N \, \psi_j^p\,  \mod{a}^p \wrt \mu \\ 
& \leq  \, \normto{a}{p}{p}  \\  
& \leq  \, \mu(B)^{-p/p'};
\end{aligned}
$$
we have used the fact that $0\leq \psi_j\leq 1$, that $p>1$, and that $\sum_{k=1}^N \psi_j = 1$ on $B$
in the first inequality above and the size condition of the $p$-atom $a$ in the second.  
By combining the preceding estimates, we obtain that there exists a constant $C$, depending on $c$
and on $p$, such that 
$$
\norm{a}{\frh^{1,p}_c} 
\leq C \, \mu(B)^{-1/p'} \,  \mu\big(B_{b+2c}(c_B)\big)^{1/p'}.
$$
Since $B$ and $B_{b+2c}(c_B)$ have the same centre, 
$$
\mu\big(B_{b+2c}(c_B)\big) 
\leq D_{1+2(c/b),b} \, \mu(B),
$$ 
whence $\norm{a}{\frh^{1,p}_c} \leq C \, D_{1+2(c/b),b}^{1/p'}$, as required.  
\end{proof}

\begin{proposition}\label{p:b}
Suppose that $p$ is in $(1,\infty]$ and that $b>c>0$. 
A function $f$ is in $\frh^{1,p}_c(M)$ if and only if $f$ is in $\frh^{1,p}_b(M)$. 
Furthermore there exist positive constants $C_1$ and $C_2$, depending on $b$, $c$ and $p$, such that  
$$
C_1\, \norm{f}{\frh^{1,p}_b} 
\leq \norm{f}{\frh^{1,p}_c} 
\leq C_2 \,  \norm{f}{\frh^{1,p}_b} 
\quant f \in \frh^{1,p}_c(M).
$$
\end{proposition}

\begin{proof}
We begin by showing that $\frh^{1,p}_c(M)\subset \frh^{1,p}_b(M)$, and that the left hand 
inequality holds.  If $a$ is a $p$-atom at scale $c$ with support contained in $B$, then 
$
\ds
{a} \, \Big[\frac{\mu (B)}{\mu \big((b/c)B\big)}\Big]^{1/p'} 
$
is a $p$-atom at scale $b$, and
$$
\begin{aligned}
\norm{a}{\frh^{1,p}_b} 
& \leq \Big[\frac{\mu \big((b/c)B\big)}{\mu (B)}\Big]^{1/p'} \\
& \leq D_{b/c,c}^{1/p'}.
\end{aligned}
$$
This implies that if $f$ belongs to $\frh^{1,p}_c(M)$, 
then $f$ is in $\frh^{1,p}_b(M)$ 
and $\norm{f}{\frh^{1,p}_b} \leq D_{b/c,c}^{1/p'} \, \norm{f}{\frh^{1,p}_c}$.

The reverse inclusion follows directly from Lemma~\ref{l: economical decomposition}.
\end{proof}

\begin{remark} \label{rem: independence h1}
Suppose that $p$ is in $(1,\infty]$. Then for every $b$ and $c$ such that 
$b>c>0$ the spaces $\frh^{1,p}_b(M)$ and $\frh^{1,p}_c(M)$ 
are isomorphic (in fact, they contain the same functions) by Proposition~\ref{p:b}. 
Hereafter we denote the space $\frh^{1,p}_1(M)$,
endowed with any of the equivalent norms defined above, simply by $\frh^{1,p}(M)$.
\end{remark}

In Section \ref{sec:5} we shall prove that $\frh^{1,p}(M)$ does not depend 
on the parameter $p$ in $(1, \infty)$, and then we shall denote all
the spaces $\frh^{1,p}(M)$ simply by $\frh^1(M)$.

\section{The local ionic space $\frh^1_I(M)$} \label{sec:3}

In this section we show that $\ghu{M}$ admits a ``ionic decomposition''.
Specifically, we shall define a ``ionic'' Hardy space $\ghuI{M}$. 
The space $\ghuI{M}$ is defined much as $\ghu{M}$, but with ions in place
of atoms. It will be clear from the definition that every atom is
an ion, but not conversely.  In fact, we shall consider a one-parameter
family of different types of ions.  When this parameter is equal to one, 
and $M$ is a Riemannian manifold with strongly bounded geometry, 
then $\ghuI{M}$ is the local Hardy space introduced by Taylor in \cite{T2}.  

\begin{definition} \label{d: ion}
Suppose that $p$ is in $(1,\infty]$ and that $\al$ is a positive real number.  
A $(p,\al)$-\emph{ion} is a function $g$ in $L^1(M)$ supported in a ball~$B$ 
with the following properties:
\begin{enumerate}
\item[\itemno1]
$\norm{g}{\infty}\leq \mu (B)^{-1}$ if $p=\infty$ and 
$\norm{g}{p}  \leq \mu (B)^{-1/p'}$ if $p\in(1, \infty)$; 
\item[\itemno2]
$\Bigmod{\ds \int_B g \wrt \mu} \leq r_B^\al$.
\end{enumerate}
A $(p,1)$-ion will be simply called a $p$-ion.
\end{definition}

\noindent
Note that Taylor considered $\infty$-ions only. 

\begin{definition} \label{d: ionic}
Suppose that $b$ and $\al$ are positive real numbers.  
The \emph{local ionic Hardy space} $\frh^{1,p,\al}_{I,b}({M})$ is the 
space of all functions~$f$ in $L^1(M)$
that admit a decomposition of the form
\begin{equation} \label{f: ionicdecomposition}
f = \sum_{j=1}^\infty \mu_j \, g_j,
\end{equation}
where the $g_j$'s are $(p,\al)$-ions \emph{supported in balls of radius at most b} 
and $\sum_{j=1}^\infty \mod{\mu_j} < \infty$.
The norm $\norm{f}{\frh^{1,p,\al}_{I,b}}$
of $f$ is the infimum of $\sum_{j=1}^\infty \mod{\mu_j}$
over all decompositions (\ref{f: ionicdecomposition}) of $f$.

If $\al=1$, then we denote 
$\frh^{1,p,\al}_{I,b}({M})$ simply by $\frh^{1,p}_{I,b}({M})$.
\end{definition}

\noindent
We shall prove that the spaces $\frh^{1,p,\al}_{I,b}({M})$ do not depend on 
$\al$. Indeed, we shall show that all these spaces 
coincide with the atomic spaces $\frh^{1,p}_{b}(M)$ and that 
the corresponding norms are equivalent.  We shall make use of the following remark.

\begin{remark}\label{r:geq1}
If $\al \geq 1$, then it is easy to show that 
$\frh^{1,p,\al}_{I,b}({M}) \subset \frh^{1,p}_{I,b}({M})$ and 
$\norm{f}{\frh^{1,p}_{I,b}}\leq \norm{f}{\frh^{1,p,\al}_{I,b}}$
for every $f$ in $\frh^{1,p,\al}_{I,b}(M)$.

Indeed, consider a $(p,\al)$-ion $g$ supported in a ball $B$.
If $r_B\geq 1$, then the size condition implies that  
$$
\Bigl| \ds \int_B g \wrt \mu \Bigr| 
\leq \norm{g}{p} \, \mu (B)^{1/p'} \leq 1\leq r_B.
$$
If $r_B<1$, then 
$$
\Bigl| \ds \int_B g \wrt \mu \Bigr| \leq r_B^{\al}\leq r_B.
$$
Hence $g$ is a $p$-ion. 
The inclusion $\frh^{1,p,\al}_{I,b}({M})\subset \frh^{1,p}_{I,b}({M})$ 
and the desired norm inequality follow.
\end{remark}

\begin{theorem} \label{t: equiv ionicalpha}
Suppose that $p\in (1,\infty]$, $\al>0$ and $b>0$.
The spaces $\frh^{1,p,\al}_{I,b}({M})$ and $\frh^{1,p}_{b}({M})$ coincide. 
Furthermore, there exist a constant $C_1$, depending on $b$ and $\al$, and 
a constant $C_2$, depending on~$b$,~$\al$ and $p$, such that 
$$
C_1  \bignorm{f}{\frh^{1,p,\al}_{I,b}} 
\leq \bignorm{f}{\frh^{1,p}_{b}} 
\leq C_2 \, \bignorm{f}{\frh^{1,p,\al}_{I,b}} 
\quant f \in \frh^{1,p}_b(M).
$$
\end{theorem}

\begin{proof}
First we prove that $\frh^{1,p}_{b}({M}) \subset \frh^{1,p,\al}_{I,b}({M})$, 
by showing that each $p$-atom at scale $b$ is a multiple of a 
$(p,\al)$-ion supported in the same ball.
Indeed, clearly each standard $p$-atom is a $(p,\al)$-ion. 
Now, suppose that $a$ is a global $p$-atom supported in a ball $B$ of radius $b$.
Then the size condition implies that
\begin{equation}\label{e:1}
\Bigl| \ds \int_B a \wrt \mu \Bigr| \leq \bignorm{a}{p}\, \mu (B)^{1/p'} \leq 1.
\end{equation} 
If $b \geq 1$, then $\Bigl| \ds \int_B a \wrt \mu \Bigl| \leq b^{\al}$ 
and $a$ is a $(p,\al)$-ion at scale $b$.
If $b<1$, then it is clear that $b^{\al}\, a$
is a $(p,\al)$-ion at scale $b$. 
Therefore $\bignorm{a}{\frh^{1,p,\al}_{I,b}} \leq 1/b^{\al}$.
Thus, $\frh^{1,p}_b({M}) \subset \frh^{1,p,\al}_{I,b}(M)$ and 
$$
\bignorm{f}{\frh^{1,p,\al}_{I,b}} 
\leq \max(1, b^{-\al}) \,\bignorm{f}{\frh^{1,p}_b}
\quant f \in \frh^{1,p}_b({M}).
$$

To prove the reverse inclusion, let $g$ be a $(p,\al)$-ion with support contained in $B$, with $r_B\leq b$.
We write $g=a+h$, where 
$$
a
= g-\frac{\chi_B}{\mu(B)}\ds \int_B g \wrt \mu
\qquad\hbox{and} \qquad
h
= \frac{\chi_B}{\mu(B)}\ds \int_B g \wrt \mu.
$$
Observe that $a$ is a multiple of a standard $p$-atom at scale $b$. 
Indeed, $\int_B a \wrt \mu  = 0$ and 
\begin{align*}
\bignorm{a}{p} 
& \leq \bignorm{g}{p} + \Bigl|\int_B g \wrt \mu \Bigr| \, \, \frac{\bignorm{\chi_B}{p}}{\mu(B)} \\
& \leq \mu (B)^{-1/p'}+ r_B^{\al}\,  \mu (B)^{-1/p'} \\
& \leq (1+b^{\al})\, \mu (B)^{-1/p'},
\end{align*}
so that  $\bignorm{a}{\frh^{1,p}_b} \leq 1+b^{\al}$.
Now, if $r_B = b$, then 
$$
\begin{aligned}
\bignorm{h}{p}
& =      \mu(B)^{-1/p'} \, \Bigmod{\int_B g \wrt \mu} \\
& \leq   \mu(B)^{-1/p'} \, b^\al, 
\end{aligned}
$$
so that $b^{-\al} \, h$ is a global $p$-atom at scale $b$, whence
$\bignorm{h}{\frh^{1,p}_b} \leq b^\al$, and $\bignorm{g}{\frh^{1,p}_b} \leq 1+2b^{\al}$.  
If, instead, $r_B < b$, then we decompose $h$ as a finite combination of $\frh^{1,p}_b$-atoms
as follows.  
Set $N:=\left[\log_2(b/r_B)\right]$ and write $h=\sum_{i=1}^{N+2} h_i$, where
$$
h_i
= \Bigl[\frac{\chi_{2^{i-1}B}}{\mu(2^{i-1}B)}-\frac{\chi_{2^{i}B}}{\mu(2^{i}B)}
   \Bigr] \, \Big(\int_B g \wrt \mu\Big) \qquad i=1,...,N+1
$$
and
$
\ds h_{N+2}
= \frac{\chi_{2^{N+1}B}}{\mu(2^{N+1}B)}\, \Big(\int_B g \wrt \mu\Big).
$
A straightforward computation shows that for all $i=1,...,N+1$,
$$
\begin{aligned}
\int_M \Bigmod{\frac{\chi_{2^{i-1}B}}{\mu(2^{i-1}B)} 
-  \frac{\chi_{2^{i}B}}{\mu(2^{i-1}B)}}^p \wrt \mu
& =\frac{\mu(2^iB \setminus 2^{i-1}B)}{\mu (2^iB)^p}
   + \frac{\mu(2^{i-1}B)}{\mu (2^{i-1}B)^p} \\
& \leq \mu(2^iB)^{1-p} + \mu (2^{i-1}B)^{1-p} \\
& \leq 2 \, \mu (2^{i-1}B)^{1-p}.
\end{aligned}
$$
Therefore
\begin{align*}
\bignorm{h_i}{p} 
& \leq r_B^\al\, 2^{1/p} \,  \mu (2^{i-1}B)^{-1/p'}\\ 
& \leq r_B^\al\, 2^{1/p} \,  D_{2,b/2}^{1/p'}\, \mu (2^{i}B)^{-1/p'};
\end{align*}
the last inequality follows from the estimate $\mu\big(2^iB\big) \leq D_{2,b/2} \, \mu\big(2^{i-1}B\big)$. 
Since $2^{i}B \in \cB_b$ for all $i=1,...,N$, 
$
{h_i}/[2^{1/p} \, D_{2,b/2}^{1/p'} \, r_B^\al]
$
is a standard $p$-atom, so that 
$\bignorm{h_i}{\frh^{1,p}_b} \leq 2^{1/p} \, D_{2,b/2}^{1/p'} \, r_B^\al$.
 
Furthermore, the functions $h_{N+1}$ and $h_{N+2}$ are supported in 
the ball $2^{N+1}B$, which has radius $\leq 2b$.  Denote by $B'$ the ball with the same
centre as $B$ and radius $2b$.  Then 
$$
\begin{aligned}
\bignorm{h_{N+1}}{p}
& \leq 2^{1/p}\,r_B^{\al}\,\mu(2^{N+1}B)^{-1/p'} \\
& \leq 2^{1/p}\,r_B^{\al}\, \frac{\mu(B')^{1/p'}}{\mu(2^{N+1}B)^{1/p'}} \, \mu(B')^{-1/p'} \\
& \leq D_{2,2b}^{1/p'} \, 2^{1/p}\,r_B^{\al}\,  \mu(B')^{-1/p'},
\end{aligned}
$$
and, similarly, 
$
\bignorm{h_{N+2}}{p}\leq D_{2,2b}^{1/p'} \,r_B^{\al}\,\mu(B')^{-1/p'}.
$ 
Thus,
$$
\bignorm{h_{N+1} + h_{N+2}}{\frh^{1,p}_{2b}}
\leq \big(2^{1/p} + 1\big) \, r_B^\al.     
$$
Then, by Proposition~\ref{p:b}, there exists a constant $C$, depending on $b$ and $p$, such that 
$$
\bignorm{h_{N+1} + h_{N+2}}{\frh^{1,p}_{b}}
\leq C \, r_B^\al.     
$$
By combining these estimates wesee that there exists a constant $C$, which depends on $b$ and $p$ 
such that 
$$
\bignorm{h}{\frh^{1,p}_b}
\leq \sum_{i=1}^{N+2}\bignorm{h_i}{\frh^{1,p}_b}
\leq C \, r_B^\al \, \big[N + 1\big].
$$
Now, observe that 
$$
r_B^\al \, N 
= r_B^\al \, \log(b/r_B) 
\leq  r_B^\al \, \big[\log_2 b- \log r_B\big] 
\leq  r_B^\al \, \log_2 b 
\leq  b^\al \, \log_2 b.
$$
Hence
$
\bignorm{h}{\frh^{1,p}_b}
\leq C,
$
so that each $(\al,p)$-ion $g$ is in $\frh^{1,p}_b(M)$ and 
$\bignorm{g}{\frh^{1,p}}\leq C$, 
where the constant $C$ depends only on $b$, $\al$ and $p$, as required.
\end{proof}

We have already mentioned that the spaces $\frh^{1,p}_{b}(M)$ will be proved to be 
independent of the parameters $p$ and $b$.  Then, by Theorem~\ref{t: equiv ionicalpha},
for $p$ in $(1,\infty]$, $b>0$ and $\al$ in $\BR^+$, the spaces
$\frh^{1,p,\al}_{I,b}(M)$ coincide with equivalence of the norms.

\begin{remark}
We shall denote by $\frh^{1}_{I}(M)$ all the spaces 
$\frh^{1,p,\al}_{I,1}(M)$, endowed with 
any of the equivalent norms defined above. 
\end{remark}

\section{The space $\gbmo{M}$}
\label{sec:4}

Suppose that $q$ is in $[1,\infty)$ and $b$ is in $\BR^+$. 
For each locally integrable function $f$ define the 
\emph{local sharp maximal function} $f^{\sharp,q}_{b}$ by
$$
f^{\sharp,q}_{b}(x)
= \sup_{B \in \cB_{b}(x)} \Bigl(\frac{1}{\mu(B)}
\int_B \mod{f-f_B}^q \wrt\mu \Bigr)^{1/q}
\quant x \in M,
$$
where $f_B$ denotes the average of $f$ over $B$ and $\cB_{b}(x)$ denotes the family
of all balls in $\cB_{b}$ centred at the point $x$. 
Define also the \emph{modified local sharp maximal function} $N^q_b(f)$ by
$$
N^q_b(f)(x)
:= f^{\sharp,q}_{b}(x) +  \Bigl[\frac{1}{\mu(B_b(x))} 
     \int_{B_b(x)} \mod{f}^q \wrt\mu \Bigr]^{1/q}
\quant x \in M, 
$$
where $B_b(x)$ denotes the ball with centre $x$ and radius $b$.
Denote by $\mathfrak{bmo}^q_b(M)$ the space of all locally integrable 
functions~$f$ such that $N^q_b(f)$ is in $\ly{M}$, endowed with the norm 
$$
\bignorm{f}{\mathfrak{bmo}^q_b}
= \bignorm{N^q_b(f)}{\infty}.
$$

The space $\mathfrak{bmo}^q_b(M)$ is related to the space $BMO^q_b(M)$,
introduced in \cite{CMM1}.  The latter is the Banach space of all locally 
integrable functions $f$ (modulo constants) such that 
$$
\bignorm{f}{BMO^q_b}=\bignorm{f^{\sharp,q}_{b}}{\infty}< \infty.
$$
As shown in \cite{CMM1}, the spaces $BMO^q_b(M)$ do not depend on the 
parameters $q$ and $b$ and we denote them all by $BMO(M)$.

\begin{remark}
Given $f$ in $\mathfrak{bmo}^q_b(M)$, we have
$$
\bignorm{f^{\sharp,q}_{b}}{\infty}
\leq \bignorm{N^q_b(f)}{\infty}
=    \bignorm{f}{\mathfrak{bmo}^q_b}.
$$ 
Denote by $[f]$ the equivalence class in $BMO^q_b(M)$ which contains $f$.
By the estimate above, the linear map 
$\iota: \mathfrak{bmo}^q_b(M)\to BMO^q_b(M)$, defined by
$
\iota (f) = [f],
$
is continuous, i.e., 
\begin{equation} \label{f: iota}
\norm{\iota(f)}{BMO^q_b} \leq \norm{f}{\mathfrak{bmo}^q_b}  
\quant f \in \mathfrak{bmo}^q_b(M).
\end{equation}
\end{remark}

%

\noindent
In the following proposition we show that 
the space $\mathfrak{bmo}^q_b(M)$ does not depend on the parameters $b$ and $q$
in the appropriate ranges.

\begin{proposition}
Suppose that $q$ is in $[1, \infty)$ and that $b>c>0$. 
The following hold:
\begin{enumerate}
\item[\itemno1]
$\mathfrak{bmo}^q_b(M)$ and $\mathfrak{bmo}^q_c(M)$ coincide 
and their norms are equivalent;
\item[\itemno2]
$\mathfrak{bmo}^q_1(M)$ and $\mathfrak{bmo}^1_1(M)$ coincide and their norms are equivalent.  
\end{enumerate}
\end{proposition}

\begin{proof}
First we prove \rmi.
Suppose that $f$ is in $\mathfrak{bmo}^q_b(M)$. 
Since $c<b$, $f^{\sharp,q}_{c}(x) \leq f^{\sharp,q}_{b}(x)$. 
Moreover, for each $x \in M$
$$
\begin{aligned}
 \frac{1}{\mu(B_c(x))}\int_{B_c(x)} \mod{f}^q \wrt\mu 
&\leq \frac{1}{\mu(B_c(x))}\int_{B_b(x)} \mod{f}^q \wrt\mu\\
&=\frac{\mu(B_b(x))}{\mu(B_c(x))} \, \frac{1}{\mu(B_b(x))}
     \int_{B_b(x)} \mod{f}^q \wrt\mu \\
&\leq D_{b/c,c}\,  \frac{1}{\mu(B_b(x))}\int_{B_b(x)} \mod{f}^q \wrt\mu, 
\end{aligned}
$$
(see (\ref{f: doubling Dtau})).
Therefore $N^q_c(f)(x)\leq D_{b/c,c}\, N^q_b(f)(x)$. 
Thus $f$ is in $\mathfrak{bmo}^q_c(M)$ and 
$\norm{f}{\mathfrak{bmo}^q_c} \leq D_{b/c,c}^{1/q}\, \norm{f}{\mathfrak{bmo}^q_b}$.

\medskip
To prove the reverse inequality, observe that,  by \cite[Prop~5.1]{CMM1},
there exists a constant $C_1$, depending only on $b$, $c$ and $M$, 
such that 
$$ 
\norm{f^{\sharp,q}_{b}}{\infty}
\leq C_1\ \norm{f^{\sharp,q}_{c}}{\infty}
\quant f \in \mathfrak{bmo}^q_c(M).
$$
Now suppose that $B_b$ is a ball of radius $b$. Then
$$
\Big[\frac{1}{\mu(B_b)}\int_{B_b} \mod{f}^q \wrt\mu \Big]^{1/q}
= \frac{1}{\mu(B_b)^{1/q}} \sup_{\norm{\phi}{L^{q'}(B_b)}\leq 1}
     \, \Bigmod{\int_{B_b} f\, \phi \wrt \mu},
$$
where $q'$ is the exponent conjugate to $q$.
If $\phi$ is a function in $L^{q'}(B_b)$ with $\norm{\phi}{L^{q'}(B_b)}\leq 1$, 
then $\mu(B_b)^{-1/q}\phi$ is a $q'$-global atom at scale $b$.
Therefore, by Lemma \ref{l: economical decomposition}, 
there exist $q'$-global atoms  $a_1, \ldots, a_N$ at scale $c$
supported in balls $B_j$ such that $\mu(B_b)^{-1/q}\phi = \sum_{j=1}^N \la_j a_j$, 
with 
$$
\sum_j \, |\la_j|
\leq C, 
$$ 
where $C$ depends only on $b$, $c$ and $p$. 
Thus, by H\"{o}lder's inequality,
$$
\begin{aligned}
\frac{1}{\mu(B_b)^{1/q}} \Bigmod{\int_{B_b} f\ \phi \wrt \mu} &=\Bigmod{\sum_{j=1}^N \la_j\int_{B_j} f\ a_j \wrt \mu}\\
&\leq \sum_{j=1}^N \, \mod{\la_j} \, \Big[ \int_{B_j} \mod{f}^q \wrt\mu \Big]^{1/q} 
        \, \norm{a_j}{q'}\\
&\leq \sum_{j=1}^N\,  \mod{\la_j} \, \Big[ \frac{1}{\mu(B_j)}\int_{B_j} \mod{f}^q 
        \wrt\mu \Big]^{1/q}\\
&\leq  C \,  \bignorm{f}{\mathfrak{bmo}^q_c}.   
\end{aligned}
$$
The above estimates imply that
$\bignorm{f}{\mathfrak{bmo}^q_b} \leq (C_1+CN)\ \bignorm{f}{\mathfrak{bmo}^q_c}$,
as required to conclude the proof of \rmi.

Next we prove \rmii.  
Recall 
that the spaces $BMO^1(M)$ and $BMO^q(M)$ agree (with equivalence of norms)
for all $q$ in $(1,\infty)$ \cite[Corollary~5.5]{CMM1}.  
Therefore there exists a constant $C$ such that   
$$
\bignorm{\iota(f)}{BMO^q}
\leq C\ \bignorm{\iota(f)}{BMO^1}
\leq C\ \bignorm{f}{\mathfrak{bmo}^1}
\quant f \in \mathfrak{bmo}^1(M),
$$
where the last inequality follows from \eqref{f: iota}.  
Thus, 
$$
\Big[\frac{1}{\mu(B)}\int_B \mod{f-f_B}^q \wrt\mu \Big]^{1/q} 
\leq C\, \bignorm{f}{\mathfrak{bmo}^1}   
\quant f \in \mathfrak{bmo}^1(M)
\quant B \in \cB_{1}.   
$$
Now suppose that $B_1$ is a ball of radius $1$. By the triangle inequality
$$
\begin{aligned}
\Big[ \frac{1}{\mu(B_1)}\int_{B_1} \mod{f}^q \wrt\mu \Big]^{1/q} 
&\leq \Big[\frac{1}{\mu(B_1)}\int_{B_1} \mod{f-f_{B_1}}^q \wrt\mu \Big]^{1/q}
     +\mod{f_{B_1}} \\
& \leq C\, \norm{f}{\mathfrak{bmo}^1}+\frac{1}{\mu(B_1)}\int_{B_1} \mod{f} \wrt\mu\\
& \leq (C+1)\, \norm{f}{\mathfrak{bmo}^1}. 
\end{aligned}
$$
These estimates imply that 
$$
\norm{f}{\mathfrak{bmo}^q} 
\leq (2C+1) \norm{f}{\mathfrak{bmo}^1}  
\quant f \in \mathfrak{bmo}^1(M),   
$$
whence $\mathfrak{bmo}^1(M)\subseteq \mathfrak{bmo}^q(M)$.

To prove the reverse containment, observe that, by H\"older's inequality,
$$
N_1^1(f)(x) \leq N_1^q(f)(x)  \quant x \in M,
$$
so that $\norm{f}{\mathfrak{bmo}^1} \leq\norm{f}{\mathfrak{bmo}^q}$, 
and $\mathfrak{bmo}^q(M)\subseteq \mathfrak{bmo}^1(M)$.

The proof of \rmii\ is complete. 
\end{proof}

\begin{remark} \label{rem: independence bmo}
In view of the observation above, 
all the spaces $\mathfrak{bmo}_b^q(M)$, $b>0$, $q$ in $[1,\infty)$, coincide.
We shall denote them simply by $\mathfrak{bmo}(M)$, 
endowed with any of the equivalent norms $\norm{\cdot}{\mathfrak{bmo}^q_b}$. 
This remark will be important in the proof of the duality 
between $\frh^{1}(M)$ and $\mathfrak{bmo}(M)$.
\end{remark}

%

\section{Duality} \label{sec:5}

In this section we shall prove that the topological dual of $\frh^{1,p}(M)$ 
is isomorphic to $\mathfrak{bmo}^{p'}(M)$, where $p'$ denotes the index conjugate to $p$.
In view of Remark~\ref{rem: independence bmo}, we consider $\mathfrak{bmo}^{p'}(M)$
endowed with the norm $\bignorm{N_1^{p'}(\cdot)}{\infty}$. Similarly, in view of Remark~\ref{rem: 
independence h1}, we may, and shall, consider $\frh^{1,p}(M)$, endowed with the $\frh_1^{1,p}(M)$-norm.  

We need more notation and some preliminary observations. 
Suppose that $p$ is in $[1,\infty)$. For each closed ball $B$ in $M$,
we denote by $\lp{B}$ the space of all functions in
$\lp{M}$ which are supported in $B$.
The union of all spaces $L^p(B)$ as $B$ varies
over all balls coincides with the space $L^p_c(M)$ of
all functions in $\lp{M}$ with compact support.
Fix a reference point $o$ in $M$ and for each positive integer $k$
denote by $B_k$ the ball centred at~$o$ with radius~$k$. 
A convenient way of topologising $L^p_c(M)$
is to interpret $L^p_c(M)$ as the strict inductive limit
of the spaces $L^p_c(B_k)$ (see \cite[II, p.~33]{Bou} for the definition of
the strict inductive limit topology).  
We denote by $X^p$ the space $L^p_c(M)$ with this topology,
and write~$X_k^p$ for $L^p_c(B_k)$.

We denote by $\frh^{1,p}_{\textrm{fin}}(M)$ the subspace of $\frh^{1,p}(M)$ 
consisting of all finite linear combinations of $p$-atoms. 
Clearly, $\frh^{1,p}_{\textrm{fin}}(M)$ is dense in $\frh^{1,p}(M)$ 
with respect to the norm of $\frh^{1,p}(M)$.  A natural norm
on $\frh^{1,p}_{\textrm{fin}}(M)$ is defined as follows:
\begin{equation} \label{f: huqfin norm}
\norm{f}{\frh^{1,p}_{\textrm{fin}}}
= \inf \Bigl\{ \sum_{j=1}^N \mod{c_j}: f = \sum_{j=1}^N c_j \, a_j,
\, \, 
\hbox{$a_j$ is a $p$-atom, $N\in\BN^+$} \Bigr\}.
\end{equation}
Note that the infimum is taken over \emph{finite} linear combinations of atoms. 
Obviously, 
\begin{equation} \label{f: hu-huqfin}
\norm{f}{\frh^{1,p}} 
\leq  \norm{f}{\frh^{1,p}_{\textrm{fin}}}
\quant f \in \frh^{1,p}_{\textrm{fin}}(M).
\end{equation}
%
\begin{remark}
Observe also that $\frh^{1,p}_{\textrm{fin}}(M)$ and $L^p_c(M)$ agree as vector spaces. 
Indeed, on the one hand each function in $\frh^{1,p}_{\textrm{fin}}(M)$ has finite 
$L^p$-norm and is compactly supported, hence it belongs to $L^p_c(M)$. 
On the other hand, suppose that $g$ is in $L^p_c(M)$ and denote by $B$ 
a ball of radius $\geq 1$ that contains the support of $g$.  Then 
$a:= \normto{g}{p}{-1} \mu(B)^{-1/p'}g$ is a global $p$-atom at scale $r_B$, which, 
by Lemma \ref{l: economical decomposition}, may be written as a finite 
linear combination of global $p$-atoms at scale 1. 
Therefore $a$ is in $\frh^{1,p}_{\textrm{fin}}(M)$, whence so is $g$.
\end{remark}


\noindent
Define
$$
f^{s,q}(x)
:= \sup_{B \in \cB_{1}(x)} \inf_{c \in \BC} \Big[\frac{1}{\mu(B)}
\int_B \mod{f-c}^q \wrt\mu \Big]^{1/q}
\quant x \in M. 
$$
It is straightforward to check that 
$f^{s,q}(x) \leq f^{\sharp,q}(x) \leq 2\,f^{s,q}(x)$ for all $x$ in $M$.
Thus, 
$$
\norm{f^{s,q}}\infty + \sup_{x \in M} \,  \Bigl[\frac{1}{\mu(B_1(x))} 
     \int_{B_1(x)} \mod{f}^q \wrt\mu \Bigr]^{1/q}
$$
is an equivalent norm on $\mathfrak{bmo}^q(M)$. 
We shall write $f^s$, instead of $f^{s,1}$.

\begin{lemma}  \label{l: analogue CW}
If $f \in \mathfrak{bmo}^q(M)$, 
then $|f| \in \mathfrak{bmo}^q(M)$ and 
$\bignorm{|f|}{\mathfrak{bmo}^q} \leq  2 \bignorm{f}{\mathfrak{bmo}^q}$.
\end{lemma}

\begin{proof}
Indeed,
\begin{align*}
\mod{f}^{\sharp,q}(x) &\leq 2\ \mod{f}^{s,q}(x) \\
& \leq 2\,\sup_{B \in \cB_{1}(x)} \Big[\frac{1}{\mu(B)}
\int_B \Bigmod{\mod{f}-\mod{f_B}}^q \wrt\mu \Big]^{1/q}\\
& \leq 2\, \sup_{B \in \cB_{1}(x)} \Big[\frac{1}{\mu(B)}
\int_B \mod{f-f_B}^q \wrt\mu \Big]^{1/q}  \\
& =    2\ f^{\sharp,q}(x),
\end{align*}
whence 
\begin{align*}
N^q(\mod{f})(x)
& =    \mod{f}^{\sharp,q}(x)
           + \Big[\frac{1}{\mu(B_1(x))}\int_{B_1(x)} \mod{f}^q \wrt\mu\Big]^{1/q}\\
& \leq 2\, f^{\sharp,q}(x)
           + \Big[\frac{1}{\mu(B_1(x))}\int_{B_1(x)} \mod{f}^q \wrt\mu\Big]^{1/q}\\
& \leq 2\, N^q(f)(x),
\end{align*}
as required.
\end{proof}

\noindent
Next we identify the dual of $\ghu{M}$ with $\frb\frm\fro(M)$. The proof
follows the lines of the classical result of Coifman and Weiss \cite{CW}
in the case of spaces of homogeneous type, and of \cite{CMM1}.

\begin{theorem} \label{t: duality}
Suppose that $p$ is in $(1,\infty)$ and let $p'$ be the index conjugate to $p$. The following hold:
\begin{enumerate}
\item[\itemno1]
for every $g$ in $\mathfrak{bmo}^{p'}(M)$ the functional $F$, initially defined
on $\frh^{1,p}_{\mathrm{fin}}(M)$ by the rule
$$
F(f) = \int_{M} f\, g \wrt {\mu },
$$
has a unique bounded extension to $\frh^{1,p}(M)$.  Furthermore
$$
\opnorm{F}{}
\leq 4\ \norm{g}{\mathfrak{bmo}^{p'}},
$$
where $\opnorm{F}{}$ denotes the norm of $F$ as a continuous linear functional on $\frh^{1,p}(M)$.
\item[\itemno2]
for every continuous linear functional
$F$ on $\frh^{1,p}(M)$ there exists a function $g_F$ in $\mathfrak{bmo}^{p'}(M)$ such that
$\norm{g_F}{\mathfrak{bmo}^{p'}} \leq 3 \, \opnorm{F}{}$ and
$$
F(f) = \int_{M} f\, g_F \wrt {\mu }
\quant f \in \frh^{1,p}_{\mathrm{fin}}(M).
$$
\end{enumerate}
\end{theorem}

\begin{proof}
The proof of \rmi\ is a straightforward adaptation of the original proof of Coifman and Weiss
in the case of spaces of homogeneous type.  The argument makes use
of Lemma~\ref{l: analogue CW} above.  We omit the details.  

Next we prove \rmii.
Since $F$ is a continuous linear functional on $\frh^{1,p}(M)$, for every $p$-atom $a$
$$
\mod{F a} \leq \opnorm{F}{} \norm{a}{\frh^{1,p}} \leq \opnorm{F}{},
$$
because each $p$-atom has $\frh^{1,p}(M)$-norm at most $1$. Thus, 
$$
\sup\{ \mod{F a}: \hbox{$a$ is a $\frh^{1,p}$-atom} \}
\leq \opnorm{F}{}.
$$
If~$f$ is in $\lp{B}$, and $r_B\geq 1$, 
then $\normto{f}{p}{-1}\mu (B)^{-1/p'}f$ is a global $p$-atom at scale $r_B$. 
Then, by Lemma~\ref{l: economical decomposition}, there exists a constant $C$, 
independent of $f$, such that 
$$
\bignorm{f}{\frh^{1,p}} 
\leq C \, \mu (B)^{1/p'} \, \bignorm{f}{p},
$$ 
whence
$$
\mod{F f}
\leq C \, \opnorm{F}{}\,  \mu (B)^{1/p'} \bignorm{f}{p}.  
$$
Hence the restriction of $F$ to $X_k^p$ is a bounded linear
functional on $X_k^p$ for each $k$.  Therefore $F$ is a continuous linear
functional on $X^p$. 
Since the dual of $X^p$ is the space $L^{p'}_{\textrm{loc}}(M)$, there exists a function 
$g_F$ in $L^{p'}_{\textrm{loc}}(M)$ such that
\begin{equation} \label{f: repr funct F}
F f
= \int_M f\, g_F \wrt \mu
\quant f \in X^p.
\end{equation}
In particular, this holds whenever $f$ is a $p$-atom.  

To conclude the proof it suffices to prove that $g_F$ belongs to $\mathfrak{bmo}^{p'}(M)$ and that
\begin{equation}  \label{f: BMO I}
\norm{g_F}{\mathfrak{bmo}^{p'}}
\leq 3\, \opnorm{F}{}. 
\end{equation}
Recall that we consider $\frh^{1,p}(M)$ endowed with the $\frh_1^{1,p}(M)$ norm (see the
beginning of this section).  Thus, we need to consider only atoms with support in balls
of radius $\leq 1$.
Suppose that $B$ is a ball of radius at most $1$, and observe that
$$
\Bigl[\int_B \mod{g_F-(g_F)_B}^{p'} \wrt \mu\Bigr]^{1/p'}
= \sup_{\norm{\vp}{\lp{B}} = 1}
      \Bigmod{\int_B \vp \, \bigl(g_F-(g_F)_B\bigr) \wrt \mu}.
$$
But
$$
\begin{aligned}
\int_B \vp \, \bigl(g_F-(g_F)_B\bigr) \wrt \mu
& = \int_B \bigl(\vp-\vp_B\bigr) \, \bigl(g_F-(g_F)_B\bigr) \wrt \mu \\
& = \int_B \bigl(\vp-\vp_B\bigr) \, g_F \wrt \mu,
\end{aligned}
$$
and since $\norm{\vp}{\lp{B}} = 1$
$$
\bigmod{\vp_B}
 \leq \Bigl[\frac{1}{\mu(B)} \, \int_B \mod{\vp}^p \wrt \mu\Bigr]^{1/p} \;
 \leq \;\mu(B)^{-1/p}.
$$
Moreover, 
$$
\begin{aligned}
\norm{\vp-\vp_B}{\lp{B}}
& \leq \norm{\vp}{\lp{B}} + \mod{\vp_B} \, \mu(B)^{1/p} \\
& \leq 2,
\end{aligned}
$$
so that the function $(\vp-\vp_B)/(2\,\mu(B)^{1/p'})$ is a standard $p$-atom. 
Therefore
$$
\Bigmod{\int_B (\vp-\vp_B) \,\,  g_F \wrt \mu}
\leq 2\, \opnorm{F}{}\, \mu(B)^{1/p'}.
$$
By combining the estimates above, we conclude that for every ball $B$ of radius at most~$1$
$$
\Bigl[\frac{1}{\mu(B)} \, 
\int_B \mod{g_F-(g_F)_B}^{p'} \wrt \mu\Bigr]^{1/p'}
\leq 2\, \opnorm{F}{},
$$
Now take a ball $B$ of radius exactly equal to $1$. We have
$$
\Bigl[\int_B \mod{g_F}^{p'} \wrt \mu\Bigr]^{1/p'}
= \sup_{\norm{\vp}{\lp{B}} = 1}
      \Bigmod{\int_B \vp \, g_F \wrt \mu}.
$$
The function $\vp/\mu(B)^{1/p'}$ is a global $p$-atom at scale $1$, thus
$$
\Bigmod{\int_B \vp \,\,  g_F \wrt \mu}
\leq \opnorm{F}{}\,\mu(B)^{1/p'}.
$$
Therefore, for every ball $B$ of radius $1$
$$
\Bigl[\frac{1}{\mu(B)} \, 
\int_B \mod{g_F}^{p'} \wrt \mu\Bigr]^{1/p'}
\leq \opnorm{F}{}.
$$
Combining these estimates, (\ref{f: BMO I}) follows.
This concludes the proof of \rmii\ and of the theorem.
\end{proof}

In view of the last result, we are now able to prove that all the spaces 
$\frh^{1,p}(M)$, with $p$ in $(1,\infty)$, coincide. 
Indeed, suppose that $1<r<p<\infty$. 
Then $(\frh^{1,r}(M))^*=(\frh^{1,p}(M))^*$, since $\mathfrak{bmo}^{r'}(M)=\mathfrak{bmo}^{p'}(M)$. 
Moreover, the identity is a continuous injection of 
$\frh^{1,p}(M)$ into $\frh^{1,r}(M)$ and $\frh^{1,p}(M)$ is a dense subspace of 
$\frh^{1,r}(M)$, therefore the Hahn--Banach theorem implies that 
$\frh^{1,r}(M)=\frh^{1,p}(M)$.

\section{Estimates for the operator $N$}
\label{sec:6}

The purpose of this section is to establish a basic $\lp{M}$ estimate 
for the operator $N$, which acts on a locally integrable function $f$ by
$$
Nf(x)
=f^{\sharp}(x) + N_0f(x)  \quant x \in M, 
$$
where $f^\sharp$ is  the \emph{local centred sharp maximal function} 
given by the formula
\begin{equation} \label{e:lcsharpmax}
f^{\sharp}(x)
= \sup_{B \in \cB_{1}(x)} \frac{1}{\mu(B)}
\int_B \mod{f-f_B} \wrt\mu 
\end{equation} 
and
$$
N_0 f(x)= \frac{1}{\mu\big(B_1(x)\big)} \int_{B_1(x)} \mod{f} \wrt\mu.
$$
Note that $f^\sharp = f_1^{\sharp,1}$ in the notation of Section~\ref{sec:4}.
The main result of this section, Theorem~\ref{t:inequality} below, will be the
key to prove a basic interpolation results for $\ghu{M}$ in the next section.  

For each locally integrable function $f$, 
define the \emph{local centred Hardy-Littlewood maximal function} $\cM f$ as 
$$
\cM f(x)=\sup_{0<r \leq 1}\frac{1}{\mu(B_r(x))}\int_{B_r(x)}|f| \ d\mu.
$$
The operator $\cM$ is bounded on $\lp{M}$ for every $p \in (1,\infty]$ 
and of weak type $1$ (for the weak type estimate, just follows the lines of
the proof of the maximal inequality in \cite{NTV}). 
Clearly $Nf(x) \leq 3 \, \cM f(x)$, so that the 
$L^p$-boundedness of $\cM$ implies that for $1<p<\infty$
$$
\norm{f}{p} \geq C\, \norm{Nf}{p}
\quant f \in \lp{M}.
$$
In the next theorem  we prove a reverse inequality.

\begin{theorem}\label{t:inequality}
Suppose that $p$ is in $(1,\infty)$.  Then there exists a constant $C$ such that 
$$
\norm{f}{p} \leq C\, \norm{Nf}{p}
$$
for every
$f \in L^1_{\textrm{loc}}(M)$ such that $Nf \in \lp{M}$.
\end{theorem}

We recall \cite[Thm.~7.3]{CMM1} that if $M$ possesses the isoperimetric property IP,  
then for each $p$ in $(1,\infty)$ there exists a constant $C$ such that 
\begin{equation}\label{e:fsharp}
\norm{f}{p} \leq C\, \norm{f^{\sharp}}{p}
\quant  f \in \lp{M}.
\end{equation}
Observe that this estimate may fail if $M$ does not possess the isoperimetric property. 
For instance, (\ref{e:fsharp}) is false for $M=\BR^n$, as shown in \cite{I1}. 
The inequality in Theorem \ref{t:inequality} is weaker than (\ref{e:fsharp}),
but it does not require the IP.  

The proof of Theorem \ref{t:inequality}, which occupies the rest of this section, will make use of 
the so-called dyadic cubes introduced by G. David and M. Christ \cite{Chr,Da} 
on spaces of homogeneous type.
In fact, Christ's construction requires only the local doubling property, as
remarked in \cite{CMM1}.   For the reader's convenience, we recall the main
properties of dyadic cubes.     

\begin{theorem} \label{t: dyadic cubes}\emph{(\cite[Thm.~3.2]{CMM1})}
There exist constants $\de$ in $(0,1)$, $a_0$, $a_1$ in $\BR^+$ and 
a collection $\cQ: = \{Q_\al^k: k \in \BZ, \al \in I_k\}$ of open subsets of $M$ such that
\begin{enumerate}
\item[\itemno1]
for each $k$ in $\BZ$, the set $\bigcup_{\al} Q_\al^k$ is of full measure in $M$;
\item[\itemno2]
if $\ell \geq k$, then either $Q_\be^\ell \subset Q_\al^k$ or
$Q_\be^\ell \cap Q_\al^k = \emptyset$;
\item[\itemno3]
for each $(k,\al)$ and each $\ell < k$ there is a unique $\be$
such that $Q_\al^k \subset Q_\be^\ell$;
\item[\itemno4]
$\diam (Q_{\al}^k) \leq a_1^k$;
\item[\itemno5]
there exists a point $z_\al^k$ in $Q$ such that 
$$
B_{a_0  \de^k}(z_{\al}^k )
\subset Q_{\al}^k
\subset B_{a_1  \de^k}(z_{\al}^k).
$$
\end{enumerate}
\end{theorem}

\noindent
We shall denote by $\cQ^k$ the class of
all dyadic cubes of ``resolution'' $k$, i.e., the family of
cubes $\{Q_{\al}^k: \al \in I_k\}$. 
We shall need the following additional properties of dyadic cubes.

\begin{proposition} \label{p: further prop} \emph{(\cite[Prop.~3.4]{CMM1})} 
Suppose that $b\in \BR^+$, $\nu\in \BZ$, and let $\de$,  
$a_0$ and $a_1$ be as in Theorem~\ref{t: dyadic cubes}.  
The following hold:
\begin{enumerate}
\item[\itemno1]
suppose that $Q$ is in $\cQ^k$ for some $k\geq \nu$, and that
$B$ is a ball such that $c_B \in Q$. 
If $r_B \geq a_1\, \de^k$, then  
\begin{equation} \label{f: ineq II}
\mu(B\cap Q) = \mu(Q);
\end{equation}
if $r_B < a_1\, \de^k$, then  
\begin{equation} \label{f: ineq I}
\mu(B\cap Q) 
\geq D_{a_1/(a_0\de),\de^\nu}^{-1} \, \mu(B);
\end{equation}
\item[\itemno2]
suppose that $\tau$ is in $[2,\infty)$.  For each $Q$ in $\cQ$ the 
space $\bigl(Q,d_{\vert Q}, \mu_{\vert Q}\bigr)$
is of homogeneous type.  Denote by $D_{\tau,\infty}^Q$ 
its doubling constant (see Remark~\ref{r: geom I} for the definition).  
Then
$$
\sup \, \Bigl\{D_{\tau,\infty}^Q:  Q \in \bigcup_{k=\nu}^\infty \cQ^k \Bigr\}
\leq D_{\tau, a_1\de^\nu}\, D_{a_1/(a_0\de), \de^\nu};
$$ 
\item[\itemno3]
for each ball $B$ in $\cB_b$, let $k$ be the integer
such that $\de^k \leq r_B < \de^{k-1}$, and
and let~$\wt B$ denote the ball with
centre $c_B$ and radius $\bigl(1+a_1\bigr)\, r_B$. 
Then $\wt B$ contains all dyadic cubes in $\cQ^k$ that
intersect $B$ and 
$$
\mu(\wt B) \leq D_{1+a_1,b} \, \mu(B);
$$
\item[\itemno4]
suppose that $B$ is in $\cB_b$, and that $k$ is an integer such that
$\de^k \leq r_B < \de^{k-1}$.  Then there are at most 
$D_{(1+a_1)/(a_0\de),b}$ dyadic cubes in $\cQ^k$ that intersect $B$.
\end{enumerate}
\end{proposition}

\noindent
In particular, property \rmii\ states that, for fixed $k$, 
all the cubes in $\cQ^k$ are spaces of homogeneous type with doubling
constants uniformly bounded from above.  More precisely, 
for each cube $Q$ in $\cQ^k$
\begin{equation}\label{e:ctauk}
D_{\tau,\infty}^Q \leq C_{\tau,k}
\qquad\hbox{where}\qquad
C_{\tau,k}:=D_{\tau, a_1\de^k}\, D_{a_1/(a_0\de), \de^k}
\end{equation}

For each locally integrable function $f$ and each dyadic cube $Q$ the 
\emph{noncentred Hardy--Littlewood maximal function}
$\cM^Q f$ is defined by
\begin{equation*} \label{f: HL}
\cM^Q f(x)
= \sup_{B:B \cap Q \ni x} \frac{1}{\mu (B \cap Q)} \int_{B \cap Q} \mod{f} \wrt \mu 
\quant x \in Q,
\end{equation*}
where each $B$ is a ball in $\cB$ whose centre belongs to $Q$.  
The operator~$\cM^Q$ is bounded on $\lp{Q}$ for every $p$ in $(1,\infty]$ 
and of weak type~$1$.  Furthermore, there exists a constant $C_0$,
depending only on the doubling constant of $(Q,d_{|Q},\mu_{|Q})$,
such that for all $Q$ in $\bigcup_{k\geq \nu} \cQ^k$
\begin{equation}\label{e:weak}
\mu \bigl(\{x \in Q\, :\, \cM^Q f(x)> \la \}\bigr)
\leq \frac{C_0}{\la} \, \norm{f}{L^1(Q)}.
\end{equation}

\noindent
For each locally integrable function $f$ and each dyadic cube $Q$ 
we define the \emph{noncentred sharp maximal} function $f^{\sharp,Q}$ by
\begin{equation}\label{e:ncsharpmax}
f^{\sharp,Q}(x)= \sup_{B : B\cap Q \ni x} \frac{1}{\mu(B \cap Q)}
\int_{B \cap Q} \mod{f-f_{B \cap Q}} \wrt\mu 
\quant x \in Q,
\end{equation}
where $B$ is a ball in $\cB$ whose centre belongs to $Q$ and 
$$
f_{B\cap Q}=\frac{1}{\mu(B\cap Q)}\int_{B\cap Q}f \wrt\mu.
$$
We split the proof of Theorem~\ref{t:inequality} into a series of lemmata.  
For each $\la>0$ we define
$$
E_{\la}
:= \{x \in Q\, :\, \cM^Q f(x)> \la\},
\qquad
F_{\la}
:= \{x \in Q\, :\, f^{\sharp,Q}(x) \leq \la\}
$$ 
and $G_{\la}^{\be, \ga} :=E_{\be \la}\cap F_{\ga \la}$. 

\begin{lemma}\label{l:goodlambda}
Suppose that $k$ is in $\BZ$. 
Then there exists a constant $A$ such that for every 
$\be> 2C_{2,k}$, $\ga >0$, $f$ in $L^1_{\textrm{loc}}(M)$,  
and $Q$ in $\cQ^k$  
\begin{align*} 
\mu \bigl(E_{\be \la}\cap F_{\ga \la}\bigr)
\leq A\, \frac{\ga}{\be} \, 
\mu \bigl(E_\la\bigr)
\end{align*}
for every $\la > \frac{C_0}{\mu(Q)} \norm{f}{L^1(Q)}$, where $C_0$ is as in (\ref{e:weak}). 
\end{lemma}

\noindent
We observe that the constant $A$ in the statement above may very well depend on the resolution $k$.
This will be no problem, for in the sequel we shall mainly work with cubes with a fixed resolution.

\begin{proof}
Set 
$
\la_0
:= \frac{C_0}{\mu(Q)} \norm{f}{L^1(Q)}.
$ 
Since $\la>\la_0$, $\mu(E_{\la})<\mu(Q)$, so that $E_{\la}$ is a proper subset in $Q$. 
Since $E_{\la}$ is open and $Q$ is a space of homogeneous type, we 
can apply a Whitney type covering lemma \cite[Thm~3.2]{CW} (with $1$ in place of 
$C$ and $K$ therein), and obtain a sequence $\{B_i \cap Q\}$ of balls in $Q$, 
where $B_i \in \cB$, such that:
\begin{enumerate}
\item[\itemno1]
$E_{\la}=\bigcup_{i} (B_i \cap Q)$;
\item[\itemno2]
there exists a constant $K_0=K_0(k)$ such that no point of $E_{\la}$ belongs to 
more than $K_0$ balls $B_i \cap Q$;
\item[\itemno3]
$(3B_i \cap Q)\cap ((E_{\la})^c \cap Q) \neq \emptyset$.
\end{enumerate}
Note that $K_0$ does not depend on the particular cube $Q$ in 
$\cQ^k$ because $K_0$ depends only on the doubling constant of the space
of homogeneous type and for cubes of the same resolution the doubling 
constants are uniformly bounded from above (see \eqref{e:ctauk} above).

By assumption, $\be > 2C_{2,k} > 2$.  
Then $G_{\la}^{\be, \ga} \subset E_{\be \la} \subset E_{\la}$, so that
$$
\begin{aligned}
\mu \bigl(G_{\la}^{\be, \ga}\bigr)
& =    \mu \Bigl[G_{\la}^{\be, \ga}\cap \Bigl(\bigcup_{i}\, (B_i \cap Q)\Bigr)\Bigr]  \\
& =    \mu \Bigl[\bigcup_{i}\,(G_{\la}^{\be, \ga}\cap B_i)\Bigr]  \\
& \leq \sum_{i} \mu (G_{\la}^{\be, \ga}\cap B_i).
\end{aligned}
$$
If $G_{\la}^{\be, \ga}\cap B_i=\emptyset$ for some index $i$, 
we simply ignore the ball $B_i$; 
otherwise, there exists at least a point $y_i \in 
G_{\la}^{\be, \ga}\cap B_i$, 
whence $f^{{\sharp,Q}}(y_i)\leq \ga \la$.

\medskip
\noindent
\emph{We claim that}
$$
E_{\be \la} \cap B_i 
\subseteq
\Big\{x \in Q\, :\, \cM^Q(f \chi_{5B_i})(x)> \frac{\be\la}{C_{2,k}}\Big\}
\quant\be \geq C_{2,k}.  
$$
The claim will imply that 
$$
\mu (G_{\la}^{\be, \ga}\cap B_i)
\leq \mu(E_{\be \la} \cap B_i) 
\leq \mu\Big(\Big\{ \cM^Q(f \chi_{5B_i})> \frac{\be\la}{C_{2,k}} \Big\}\Big).
$$
To prove the claim, we consider the centred Hardy--Littlewood 
maximal function on the cube $Q$ defined by
$$
\widetilde{\cM}^Q f(x)
= \sup_{r>0} \frac{1}{\mu (B_r(x) \cap Q)} \int_{B_r(x) \cap Q} \mod{f} \wrt \mu
\quant x \in Q.
$$
Since the restriction of $\mu$ to each cube $Q$ is a doubling measure with
doubling constant bounded above by $C_{2,k}$,  
$$
\cM^Q f(x)\leq C_{2,k} \, \widetilde{\cM}^Q f(x) 
\quant x \in Q.
$$ 
Suppose that $x \in E_{\be\la} \cap B_i$ and $\be \geq C_{2,k}$.
We need to prove that 
$$
\cM^Q(f \chi_{5B_i})(x) > \frac{\be\la}{C_{2,k}}.
$$
Clearly, $\widetilde{\cM}^Q f(x) > \be\la/C_{2,k}$, 
so that there exists a ball $B_r(x)$ such that 
$$
\frac{1}{\mu (B_r(x) \cap Q)} \int_{B_r(x) \cap Q} \mod{f} \wrt \mu >\frac{\be\la}{C_{2,k} }.
$$
Condition \rmiii\ above implies that there exists a point 
$x_i$ in $3B_i \cap Q$ such that 
\begin{equation}\label{e : maximal}
\cM^Q f(x_i)\leq \la.
\end{equation}
Since we have assumed that $\be \geq C_{2,k} $, $x_i \notin B_r(x)$,
for otherwise
$$
\cM^Q f(x_i) 
\geq \frac{1}{\mu (B_r(x) \cap Q)} \int_{B_r(x) \cap Q} \mod{f} \wrt \mu 
>    \frac{\be\la}{C_{2,k} } 
\geq \la. 
$$
Since $x_i$ is in $3B_i \setminus B_r(x)$, $r < 4\, r_{B_i}$. 
Hence $B_r(x) \subset 5 B_i$ and
$$
\frac{\be\la}{C_{2,k} } 
<    \frac{1}{\mu (B_r(x) \cap Q)} \int_{B_r(x) \cap Q} \mod{f}\chi_{5B_i} \wrt \mu 
\leq \cM^Q(f\chi_{5B_i})(x).
$$
This concludes the proof of the claim.
 
Now we observe that 
$$
\cM^Q(f\chi_{5B_i})(x) \leq \cM^Q((f-f_{5 B_i\cap Q})\chi_{5B_i})(x)+\mod{f_{5 B_i\cap Q}}.
$$
Since $x_i$ is in $3B_i \cap Q$ and $\cM^Q f(x_i)\leq \la$ by (\ref{e : maximal}),  
$$
\mod{f_{5 B_i\cap Q}}\leq \frac{1}{\mu (5 B_i \cap Q)} \int_{5 B_i \cap Q} \mod{f} \wrt \mu \leq \cM^Q f(x_i) \leq \la.
$$
Therefore, if $\be > 2\,C_{2,k} $, then 
$\mod{f_{5 B_i\cap Q}} < \frac{\be}{2\,C_{2,k} }\la$. 
This estimate, together with
the weak type $1$ inequality for $\cM^Q$ and the assumption that 
$f^{\sharp,Q}(y_i) \leq \ga \la$, implies that, if $\be > 2\,C_{2,k}$, then 
\begin{align*}
\mu\Big(\Big\{ \cM^Q(f \chi_{5B_i})> \frac{\be\la}{C_{2,k}} \Big\}\Big) 
& \leq \mu\Big(\Big\{ \cM^Q((f-f_{5 B_i\cap Q})\chi_{5B_i})
      > \frac{\be\la}{2\,C_{2,k}}\Big\}\Big) \\
& \leq C_0\, \frac{2C_{2,k}}{\be\la} \int_Q \mod{f-f_{5 B_i\cap Q}}\chi_{5B_i} \wrt \mu\\
& \leq C_0\, \frac{2C_{2,k}}{\be\la}\, \mu(5B_i \cap Q)\, f^{\sharp,Q}(y_i)\\
& \leq C_0\, 2C_{2,k} \frac{\ga}{\be} \,\mu(5B_i \cap Q).
\end{align*}
Thus, we have proved that 
$$
\mu (G_{\la}^{\be,\ga}\cap B_i)
\leq C_0 \, 2C_{2,k} \,\frac{\ga}{\be} \,\mu(5B_i \cap Q),
$$
which, together with the doubling property on $Q$ and condition \rmii\ above, implies that
\begin{align*}
\mu (G_{\la}^{\be, \ga})
&\leq 2\,C_{2,k} \,C_0 \frac{\ga}{\be} \sum_i \mu(5B_i \cap Q)\\
&\leq 2\,C_{2,k} \,C_0\,C_{5,k} \frac{\ga}{\be} \sum_i \mu(B_i \cap Q)\\
&\leq 2\,C_{2,k} \,C_0\,C_{5,k}\,K_0 \frac{\ga}{\be}\, \mu(E_{\la}),
\end{align*}
as required (with $A=2\,C_{2,k} \,C_0\,C_{5,k}\,K_0$).
\end{proof}

\begin{lemma}\label{l:LpL1}
For each integer $k$ there exists a constant $C=C(k)$ such that for each cube $Q$ in $\cQ^k$
and every locally integrable function $f$
$$
\|f\|_{\lp{Q}}^p 
\leq C \, \Bigl( \|f^{{\sharp,Q}}\|_{\lp{Q}}^p + \mu(Q)^{1-p}\, \|f\|_{L^1(Q)}^p \Bigl).
$$
\end{lemma}

\begin{proof}
Since $\cM^Q f \geq |f|$ almost everywhere, 
it suffices to show that 
$$
\|\cM^Q f\|_{\lp{Q}}^p 
\leq C \, \Bigl( \|f^{{\sharp,Q}}\|_{\lp{Q}}^p + \mu(Q)^{1-p}\,\|f\|_{L^1(Q)}^p \Bigl),
$$
We set $E_{\la}=\{x \in Q\, :\, \cM^Q f(x)> \la\}$ and $\la_0=\frac{C_0}{\mu(Q)} \norm{f}{L^1(Q)}$, as in Lemma \ref{l:goodlambda}.
Note that for each $\be>0$
\begin{align*}
\|\cM^Q f\|_{\lp{Q}}^p
& = p \ioty \la^{p-1}\, \mu \bigl(E_{\la} \bigr)\wrt \la\\
& = p\, \be^p \ioty \la^{p-1}\, \mu \bigl(E_{\be \la} \bigr)\wrt \la\\
& =p\, \be^p \int_0^{\la_0} \la^{p-1}\, \mu \bigl(E_{\be \la} \bigr)\wrt \la
       + p\, \be^p \int_{\la_0}^{+ \infty} \la^{p-1}\, \mu \bigl(E_{\be \la} \bigr)\wrt \la.
\end{align*}
Denote by $I_1$ and $I_2$ the first and the second integral in the last line above,
respectively. 
Since the maximal operator $\cM^Q$ is of weak type~$1$,
\begin{align*}
I_1 
& \leq C_0\, \be^{-1} \|f\|_{L^1(Q)} \int_0^{\la_0} \la^{p-2} \wrt \la\\
& =    \,\frac{C_0^p}{p-1}\, \be^{-1} \mu(Q)^{1-p} \|f\|^p_{L^1(Q)}. 
\end{align*}
Now, we choose $\be>2C_{2,k}$. 
Given $\ga>0$, we write $I_2$ as 
$$
\int_{\la_0}^{+ \infty} \la^{p-1}\, \mu \bigl(E_{\be \la} 
\cap \{f^{\sharp,Q} \leq \ga \la \} \bigr)\wrt \la
+\int_{\la_0}^{+ \infty} \la^{p-1}\, \mu \bigl(E_{\be \la} \cap \{f^{\sharp,Q} > \ga \la\}\bigr)\wrt \la.
$$
Then, by Lemma \ref{l:goodlambda},
\begin{align*}
I_2 
& \leq \frac{A\ga}{\be}\, \int_{\la_0}^{+ \infty} \la^{p-1}\, \mu \bigl(E_{\la}) \wrt \la +
 \int_{\la_0}^{+ \infty} \la^{p-1}\, \mu \bigl(\{f^{\sharp,Q} > \ga \la\}\bigr)\wrt \la\\
&= \frac{A\ga}{\be}\, \int_{\la_0}^{+ \infty} \la^{p-1}\, \mu \bigl(E_{\la}) \wrt \la +
\frac{1}{\ga^p} \int_{\ga \la_0}^{+ \infty} \la^{p-1}\, \mu \bigl(\{f^{\sharp,Q} > \la\}\bigr)\wrt \la\\
& \leq  \frac{A\ga}{p\be}\, \|\cM^Q f \|^p_{\lp{Q}} 
         + \frac{1}{p\ga^p}\, \bignormto{f^{\sharp, Q}}{\lp{Q}}{p}.
\end{align*}
By combining the estimates above, we see that 
$$
(1-A\, \be^{p-1} \ga)\|\cM^Q f\|^p_{\lp{Q}} 
\leq \frac{C_0^p\,p}{p-1}\, \be^{p-1} \mu(Q)^{1-p} \|f\|^p_{L^1(Q)}+
 \frac{\be^p}{\ga^p} \|f^{\sharp, Q}\|^p_{\lp{Q}}.
$$
Now, we choose $\ga=1/(2 A\, \be^{p-1})$, and obtain 
$$
\begin{aligned}
\|\cM^Q f\|^p_{\lp{Q}} 
& \leq \frac{2\,C_0^p\,p}{p-1}\, \be^{p-1} \mu(Q)^{1-p} \|f\|^p_{L^1(Q)}+
      2^{p+1}\, A^p\,\be^{p^2} \|f^{\sharp, Q}\|^p_{\lp{Q}} \\
& \leq C \, \Bigl( \|f^{{\sharp,Q}}\|_{\lp{Q}}^p + \mu(Q)^{1-p}\|f\|_{L^1(Q)}^p \Bigl),
\end{aligned}
$$ 
where $C=\max \big( 2\,C_0^p\,\frac{p}{p-1}\, \be^{p-1}, 2^{p+1}\, A^p\,\be^{p^2}\big)$, 
as required.
\end{proof}

\begin{lemma}\label{l:N0}
For all integers $k$ large enough and for each cube $Q$ in~$\cQ^k$
$$
\bignorm{f}{L^1(Q)} 
\leq D_{(1+a_1\de^k)/a_0\de^k, a_0\de^k} \, \bignorm{N_0f}{L^1(Q)}.
$$ 
\end{lemma}

\begin{proof}
For the sake of definiteness, suppose that $Q$ is the dyadic cube $Q^k_{\al}$. 
Then $Q$ is contained in $B_{ a_1\de^k}(z_{\al}^k)$ by Theorem~\ref{t: dyadic cubes}~\rmv.
Denote by $\tilde{B}$ the ball $B_{1+a_1\de^k}(z_{\al}^k)$. 
Then $B_1(x) \subset \tilde{B}$ for each $x$ in $Q$.
Furthermore, if $k$ is large enough, then $B_1(x) \supset Q$ for every $x$ in $Q$.  
so that $\mu(B_1(x)) \leq \mu(\wt B)$.  Then, by Tonelli's theorem, 
\begin{align*}
\norm{N_0 f}{L^1(Q)}
& =    \int_Q \frac{\wrt \mu(x)}{\mu(B_1(x))}\int_{B_1(x)}|f(y)| \wrt \mu(y) \\
& \geq \frac{1}{\mu(\wt B)}\, \int_Q \wrt \mu(x)\int_{Q}|f(y)| \wrt \mu(y) \\
& \geq \frac{\mu(Q)}{\mu(\wt B)}\, \int_{Q}|f(y)| \wrt \mu(y).  
\end{align*}
Recall that $Q$ contains $B_{a_0\de^k}(z_\al^k)$, so that 
$$
\frac{\mu(Q)}{\mu(\wt B)}
\geq D_{(1+a_1\de^k)/a_0\de^k, a_0\de^k}^{-1}.   
$$
and the required estimate follows.
\end{proof}

\begin{lemma}\label{l:sharpQ}
Suppose that $k$ is an integer $> \left[\log_{\delta} (1/(2 a_1))\right]$, 
where $\delta$ and $a_1$ are as in Theorem~\ref {t: dyadic cubes}. 
Then there exists a constant $C$, depending on $k$, such that for each cube $Q$ in $\cQ^k$
$$
f^{\sharp,Q}(x) \leq C\, f^{\sharp}(x)
\quant x \in Q
$$
(see \eqref{e:ncsharpmax} and \eqref{e:lcsharpmax} for the definitions of 
$f^{\sharp,Q}$ and $f^{\sharp}$, respectively).  
\end{lemma}

\begin{proof}
For each $b>0$ we define the noncentred sharp function
$\widetilde{f}^{\sharp}_b$ of a locally integrable function $f$ as
$$
\widetilde{f}^{\sharp}_b(x)
= \sup_{B} \frac{1}{\mu (B)} \int_B \mod{f-f_B} \wrt \mu 
\quant x \in M,
$$
where the supremum is taken over all balls in $\cB_b $ that contain $x$.

We first show that there exists a constant $C$, depending on $k$, such that
$f^{\sharp,Q}(x) \leq C\, \widetilde{f}^{\sharp}_{a_1 \delta^k}(x)$ 
for each cube $Q$ in $\cQ^k$ and for any $x$ in $Q$. 
(see Theorem~\ref{t: dyadic cubes} for the definition of $a_1$ and $\delta$).

Choose $Q$ in $\cQ^k$. 
Take $x$ in $Q$ and suppose that $B$ is a ball whose centre belongs to $Q$ and such that $x \in B\cap Q$.
We consider the cases where $r_B < a_1 \delta^k$ and $r_B \geq a_1 \delta^k$ separately.
If $r_B < a_1 \delta^k$, the triangle inequality gives
\begin{align*} 
\frac{1}{\mu (B\cap Q)} \int_{B\cap Q} \mod{f-f_{B\cap Q}} \wrt \mu 
&\leq \frac{1}{\mu (B\cap Q)} \int_{B\cap Q} \mod{f-f_{B}} \wrt \mu
   + \mod{f_B-f_{B\cap Q}}\\
& \leq \frac{2}{\mu (B\cap Q)} \int_{B\cap Q} \mod{f-f_{B}} \wrt \mu.
\end{align*} 
By Proposition \ref{p: further prop}~\rmi, we have that 
$
\mu(B\cap Q) \geq D_{a_1/(a_0\de),\de^k}^{-1} \, \mu(B),
$ 
so that 
$$
\frac{1}{\mu (B\cap Q)} \int_{B\cap Q} \mod{f-f_{B\cap Q}} \wrt \mu 
\leq \frac{2\,D_{a_1/(a_0\de),\de^k}}{\mu (B)} \int_{B} \mod{f-f_{B}} \wrt \mu.
$$
Since the ball $B$ belongs to $\cB_{a_1 \delta^k}$, 
the right hand side of the formula above is majorised
by $2\,D_{a_1/(a_0\de),\de^k} \, \widetilde{f}^{\sharp}_{a_1 \delta^k}(x)$. 

Now assume that $r_B \geq a_1 \delta^k$. 
For the sake of definiteness, suppose that $Q$ is the dyadic cube $Q_\al ^k$.
Recall that $\diam (Q) \leq a_1 \, \delta^k$, by Theorem~\ref{t: dyadic cubes}~\rmiv,  
whence $Q\cap B=Q$.  Moreover, 
$
B_{a_0 \de^k}(z_{\al}^k)
\subset Q
\subset B_{a_1 \de^k}(z_{\al}^k),
$ 
by Theorem \ref {t: dyadic cubes}~\rmv.  
Denote by $\overline{B}$ the ball $B_{a_1 \de^k}(z_{\al}^k)$.
Then, by the triangle inequality,
\begin{align*} 
\frac{1}{\mu (B\cap Q)} \int_{B\cap Q} \mod{f-f_{B\cap Q}} \wrt \mu 
&\leq \frac{1}{\mu (B\cap Q)} \int_{B\cap Q} \mod{f-f_{\overline{B}}} \wrt \mu
        + \mod{f_{\overline{B}}-f_{B\cap Q}}\\
& \leq \frac{2}{\mu (B\cap Q)} \int_{B\cap Q} \mod{f-f_{\overline{B}}} \wrt \mu\\
& \leq \frac{2}{\mu \big(B_{a_0 \de^k}(z_{\al}^k\big))} \, 
        \int_{\overline{B}} \mod{f-f_{\overline{B}}} \wrt \mu.
\end{align*} 
Now, the local doubling property implies that 
$$
\mu(\overline{B}) \leq  D_{a_1/a_0,a_0 \de^k} \, \mu \big(B_{a_0 \de^k}(z_{\al}^k)\big);
$$
hence, the right hand side can be estimated from above by 
$$
\frac{2 \, D_{a_1/a_0,a_0 \de^k}}{\mu (\overline{B})} \int_{\overline{B}} \mod{f-f_{\overline{B}}} \wrt \mu,
$$
which, in turn, may be majorised by 
$2 \, D_{a_1/a_0,a_0 \de^k} \, \widetilde{f}^{\sharp}_{a_1 \delta^k}(x)$, 
for the ball $\overline{B}$ has radius $a_1 \, \de^k$. 
By taking the supremum over all balls $B$ containing $x$ and whose centre belongs to $Q$, we get
$$
f^{\sharp,Q}(x) \leq 2\,D_{a_1/(a_0\de),\max(1,a_0)\de^k}\, \widetilde{f}^{\sharp}_{a_1 \delta^k}(x)
\quant x \in Q.
$$
The local doubling property ensures that for each $b$ in $\BR^+$ there exists a constant~$C$ such that 
$\widetilde{f}^{\sharp}_{b} \leq C \, f^{\sharp}_{2b}$.  Therefore
$$
f^{\sharp,Q}(x) \leq C\, f^{\sharp}_{2 a_1 \delta^k}(x)
\quant x \in Q.
$$
Now, if we choose the integer $k$ large enough so that $2 a_1 \delta^k \leq 1$, i.e., 
$k>\left[\log_{\delta} (1/2 a_1)\right]$, we get
$f^{\sharp}_{2 a_1 \delta^k} \leq f^{\sharp}$, which gives the desired conclusion.
\end{proof}

\noindent
Now we are ready to prove the main result of this section.

\begin{proof}[Proof of Theorem \ref{t:inequality}]
Fix an integer $k$ so large that Lemmata~\ref{l:N0} and \ref{l:sharpQ} hold.  
In particular, $k$ must be $> \big[\log_{\de} (1/2 a_1)\big]$.  
The cubes in $\cQ^k$ are pairwise disjoint and their union is a set of full measure in $M$, so that 
\begin{align*}
\left\|f\right\|^p_{\lp{M}}
& =    \sum_{Q\in \cQ^k} \|f\|^p_{\lp{Q}}\\
& \leq C\, \sum_{Q\in \cQ^k} \, \big[ \|f^{\sharp,Q}\|^p_{\lp{Q}} 
         +  \mu(Q)^{1-p} \, \|f\|^p_{L^1(Q)} \big]\\
& \leq C\, \sum_{Q\in \cQ^k} \, \big[ \|f^{\sharp}\|^p_{\lp{Q}} 
         +  \mu(Q)^{1-p} \, \|N_0 f\|^p_{L^1(Q)}\big];
\end{align*}
the first inequality above follows from Lemma \ref{l:LpL1}, 
and the second is a consequence of Lemmata \ref{l:N0} and \ref{l:sharpQ}.
Furthermore, by H\"older's inequality,
$$
\begin{aligned}
\frac{1}{\mu(Q)^{1/p'}} \, \int_Q \mod{N_0 f} \wrt \mu 
& \leq \frac{1}{\mu(Q)^{1/p'}} \, \Big[\int_Q \mod{N_0 f}^p \wrt \mu\Big]^{1/p} 
          \, \mu(Q)^{1/p'} \\
& =    \|N_0 f\|_{\lp{Q}}.  
\end{aligned}
$$
Thus,
\begin{align*}
\left\|f\right\|^p_{\lp{M}}
&\leq C\, \big[\|f^{\sharp}\|^p_{\lp{M}} +\left\|N_0 f\right\|^p_{\lp{M}} \big]\\
& \leq C\, \left\|N f\right\|^p_{\lp{M}},
\end{align*}
as required.
\end{proof}

\section{Interpolation}
\label{sec:7}

Suppose that $X$ and $Y$ are Banach spaces, and that $\te$ is in $(0,1)$.
We denote by $S$ the strip $\{z \in \BC: \Re z\in (0,1)\}$, and by $\OV S$ its closure.
We consider the class $\cF(X,Y)$ of all functions $F: \OV S \to X+Y$ 
with the following properties:
\begin{enumerate}
\item
$F$ is continuous and bounded in $\OV S$ and analytic in $S$;
\item
the functions $t\mapsto F(it)$ and $t\mapsto F(1+it)$ are continuous from $\BR$ into $X$ and $Y$ respectively;
\item
$\lim_{|t|\rightarrow +\infty} \|F(it) \|_{X}=0$ and 
$\lim_{|t|\rightarrow +\infty} \|F(1+it) \|_{Y}=0$.
\end{enumerate}
We endow $\cF(X,Y)$ with the norm
$$
\bignorm{F}{\cF(X,Y)} 
= \sup \bigset{ \max\big( \bignorm{F(it)}{X}, \bignorm{F(1+it)}{Y}\big): t \in \BR}.  
$$ 
We define the complex interpolation space $(X,Y)_{[\te]}$~by 
$$
(X,Y)_{[\te]}=\set{F(\te)\, : \, F \in \cF(X,Y)},
$$
endowed with the norm
$$
\bignorm{f}{(X,Y)_{[\te]}}
=\inf \bigset{\bignorm{F}{\cF(X,Y)} \, : \, F\in \cF(X,Y) \, \hbox{and} \, F(\te)=f}.
$$
For more on the complex interpolation method see, for instance, \cite{BL}.

\begin{theorem} \label{t: interpolation}
Suppose that $\te$ is in $(0,1)$.  The following hold:
\begin{enumerate}
\item[\itemno1]
if $p_\te$ is $2/(1-\te)$, then
$\bigl(\ld{M},\mathfrak{bmo}(M)\bigr)_{[\te]} = L^{p_\te}(M)$;
\item[\itemno2]
if $p_\te$ is $2/(2-\te)$, then
$\bigl(\frh^1(M),\ld{M}\bigr)_{[\te]} = L^{p_\te}(M)$.
\end{enumerate}
\end{theorem}

\begin{proof}
First we prove \rmi.  Observe that 
$$
L^{p_\te}(M) 
= \bigl(\ld{M},\ly{M}\bigr)_{[\te]}
\subseteq \bigl(\ld{M},\mathfrak{bmo}(M)\bigr)_{[\te]}; 
$$ 
the containment above follows from
the fact that $\ly{M}\subset \mathfrak{bmo}(M)$. 

In order to prove the reverse inclusion,
suppose that $f$ is in the interpolation space
$\bigl(\ld{M},\mathfrak{bmo}(M)\bigr)_{[\te]}$.  Then, given $\ep>0$ there
exists a function $F$ in $\cF(L^2(M),\mathfrak{bmo}(M))$  such that $F(\te) = f$ and
$$
\|F\|_{\cF(L^2,\mathfrak{bmo})}
\leq \|f\|_{(L^2,\mathfrak{bmo})_{[\te]}}+\ep.
$$
Let $\phi$ be any measurable function which associates to any point 
$x$ in $M$ a ball $\phi(x)$ in $\cB_1(x)$. Furthermore, let 
$\eta : M \times M \to \BC$ be any measurable function with $\mod{\eta} = 1$.  
We consider the linear operators $S^{\phi,\eta}$ and $T^{\eta}$ which act on 
a function $f$ in $\ld{M}$ as follows:
$$
S^{\phi,\eta} f (x)
= \frac{1}{\mu\bigl(\phi(x)\bigr)}
\int_{\phi(x)} \bigl[ f -f_{\phi(x)}\bigr]
\, \eta(x,\cdot) \wrt \mu
\quant x \in M
$$
and
$$
T^{\eta} f (x)
= \frac{1}{\mu\bigl(B_1(x)\bigr)}
\int_{B_1(x)} f \, \eta(x,\cdot) \wrt \mu
\quant x \in M.
$$
Then
\begin{equation}
\sup_{\phi,\eta} \mod{ S^{\phi,\eta} f}
= f^\sharp
\qquad\hbox{and} \qquad
\sup_{\eta} \mod{ T^{\eta} f}
= N_0 f.
\end{equation}
For each $\phi$ and $\eta$ as before, consider the functions
$S^{\phi,\eta}F$ and $T^{\eta} F$, where $F$ is in the space $\cF(L^2(M),\mathfrak{bmo}(M))$.

We \emph{claim} that $S^{\phi,\eta}F$ and $T^{\eta}F$ belong to
the class $\cF(L^2(M),L^\infty(M))$, 
$$
\bignorm{S^{\phi,\eta}F}{\cF(L^2,L^\infty)}
\leq C \, \bignorm{F}{\cF(L^2,\mathfrak{bmo})}
$$
and
$$
\bignorm{T^{\eta}F}{\cF(L^2,L^\infty)}
\leq C \, \bignorm{F}{\cF(L^2,\mathfrak{bmo})}.
$$
Indeed, recall that $g^\sharp \leq 2\, \cM g$ and that $\cM$ is bounded
on $\ld{M}$.  Thus, 
$$
\bignorm{S^{\phi,\eta}F(it)}{2}
 \leq \|F(it)^\sharp \|_{2} 
 \leq 2 \, \|\cM F(it)\|_{2} 
 \leq C \, \|F(it)\|_{2}.
$$
Note that the constant $C$ in the above inequality does 
not depend on $\phi$ and $\eta$.  
Moreover,
$$
\|S^{\phi,\eta}F(1+it)\|_{\infty}
 \leq \|F(1+it)^\sharp \|_{\infty} 
 \leq  \|F(1+it)\|_{\mathfrak{bmo}}.
$$
Similarly,
$$
\bignorm{T^{\eta}F(it)}{2}
 \leq  \|\cM F(it)\|_{2} 
 \leq C \, \|F(it)\|_{2};
$$
and
$$
\|T^{\eta}F(1+it)\|_{\infty}
 \leq \|N_0 F(1+it) \|_{\infty} 
 \leq  \|F(1+it)\|_{\mathfrak{bmo}},
$$
where $C$ is independent of $\eta$.  Hence
$$
\begin{aligned}
\norm{S^{\phi,\eta}f}{p_{\theta}} &= \norm{S^{\phi,\eta}F(\te)}{(L^2, L^{\infty})_{[\te]}}\\
& \leq \norm{S^{\phi,\eta}F}{\cF(L^2,L^\infty)}\\
& \leq C \, \norm{F}{\cF(L^2,\mathfrak{bmo})} \\
& \leq C \bigl(\|f\|_{(L^2,\mathfrak{bmo})_{[\te]}}+\epsilon \bigr).
\end{aligned}
$$
By taking the infimum over all $\epsilon>0$ we get
$$
\norm{S^{\phi,\eta}f}{p_{\theta}} \leq C \,\|f\|_{(L^2,\mathfrak{bmo})_{[\te]}}.
$$
Now, by taking the supremum over all $\phi$ and $\eta$ 
we obtain the estimate
\begin{equation}\label{e:sharp}
\norm{f^\sharp}{p_{\theta}}
\leq C \, \|f\|_{(L^2,\mathfrak{bmo})_{[\te]}}.
\end{equation}
Similarly, we get
$$
\norm{T^{\eta}f}{p_{\theta}} \leq C \|f\|_{(L^2,\mathfrak{bmo})_{[\te]}}
$$
and taking the supremum over all functions $\eta$ we have
\begin{equation}\label{e:N0}
\bignorm{N_0 f}{p_{\theta}}
\leq C \, \|f\|_{(L^2,\mathfrak{bmo})_{[\te]}}.
\end{equation}
Now, applying Theorem~\ref{t:inequality} and combining (\ref{e:sharp}) and (\ref{e:N0}) 
we may conclude that
\begin{align*}
\bignorm{f}{p_\te} &\leq C\, \left\|N f\right\|_{p_\te}\\
&\leq C \, \big(\bignorm{f^\sharp}{p_{\theta}}+ \bignorm{N_0 f}{p_{\theta}}  \big)\\
&\leq C\, \|f\|_{(L^2,\mathfrak{bmo})_{[\te]}}
\quant f \in \bigl(\ld{M},\mathfrak{bmo}(M)\bigr)_{[\te]} 
\end{align*}
and the required inclusion
$\bigl(\ld{M},\mathfrak{bmo}(M)\bigr)_{[\te]} \subset L^{p_\te}(M)$ follows.  

To prove \rmii, we may apply a duality argument \cite[Corollary~4.5.2]{BL}.  We omit the details. 
\end{proof}

\section{On the {$\frh^1-L^1$} boundedness of operators}
\label{sec:8}

One of the reasons which make $\frh^{1}(M)$ useful is 
that to prove that a linear operator $T$ maps 
$\frh^{1}(M)$ to a Banach space $X$ it suffices
to prove that $T$ is uniformly bounded on atoms. 
This extends to the space $\frh^{1}(M)$ the analogous result for $H^1(M)$
(see \cite{MSV}).

We need more notation. Suppose that $p$ is in $[1,\infty)$. For each closed ball $B$ in $M$,
we denote by $L^p(B)$ the space of all functions in $L^p(M)$ which are supported in $B$.
The union of all spaces $L^p(B)$ as $B$ varies
over all balls coincides with the space $L^p_c(M)$ of
all functions in $L^p(M)$ with compact support.
Fix a reference point $o$ in $M$ and for each positive integer $k$
denote by $B_k$ the ball centred at~$o$ with radius~$k$. 
A convenient way of topologising $L^p_c(M)$
is to interpret $L^p_c(M)$ as the strict inductive limit
of the spaces $L^p_c(B_k)$ (see \cite[II, p.~33]{Bou} for the definition of
the strict inductive limit topology).  
We denote by $X^p$ the space $L^p_c(M)$ with this topology,
and write~$X_k^p$ for $L^p_c(B_k)$.
It is well known that 
The topological dual of $X^p$ is $L^{p'}_{\textrm{loc}}(M)$, where $p'$ denotes the index conjugate to $p$.

Note that the spaces $X_k^p$ and $X^p$ differ from the spaces, denotes exactly in the same way,
considered in \cite{MSV}, for functions in our version of $X_k^p$ and $X^p$ need
not have vanishing integral. 

\begin{theorem}\label{t:fin}
 Suppose that $p$ is in $(1,\infty)$ and that $T$ is a $L^1(M)$-valued linear operator defined on $\frh^{1,p}_{\textrm{fin}}(M)$ 
with the property that 
$$
A:=\sup\{ \norm{T a}{1}: \hbox{\emph{$a$ is a $p$-atom}} \} 
< \infty.
$$
Then there exists a unique bounded operator $\wt T$ from $\frh^1(M)$ to $L^1(M)$ which extends~$T$.
\end{theorem}

\begin{proof}
Suppose that $B$ is a ball of radius $r_B \geq 1$. 
For each $f \in \lp{B}$ such that $\norm{f}{p}=1$ set $a=\mu (B)^{-1/p'} f$, where $p'$ denotes 
the index conjugate to $p$. 
Then $a$ is a $p$-atom at scale $r_B$ and by Lemma \ref{l: economical decomposition} there exist 
global $p$-atoms at scale $1$, $a_1, \ldots, a_N$ such that $a=\sum_{j=1}^{N}c_j a_j$, 
with $|c_j|\leq C$, where $C$ and $N$ are constants, which depend only on $r_B$ and $M$.
Thus we get
\begin{align*}
\norm{Tf}{1}&=\norm{T(\mu (B)^{1/p'} a)}{1}\\
&\leq \mu (B)^{1/p'}\sum_{j=1}^{N}|c_j| \, \norm{Ta_j}{1}\\
&\leq C\, N\, A\, \mu (B)^{1/p'} 
\end{align*}
for every $f \in \lp{B}$ such that $\norm{f}{p}=1$.

In particular, the restriction of $T$ to $X^p_k$ is bounded from $X^p_k$ to $L^1(M)$ for each~$k$. 
Therefore, $T$ is bounded from $X^p$ to $L^1(M)$.
It follows that the transpose operator $T^*$ is bounded from $L^{\infty}(M)$ to the dual of 
$X^p$, which can be identified with $L^{p'}_{\textrm{loc}}(M)$.
Therefore, for every $f$ in $L^{\infty}(M)$ and every $p$-atom $a$ we have
$$
\prodo{Ta}{f}=\prodo{a}{T^*f}=\int_{M} a\, T^*f \wrt {\mu },
$$ 
so that
\begin{equation} \label{e:est infty}
\Bigl| \int_M  a\, T^*f \wrt {\mu } \Bigr|
= |\prodo{Ta}{f}|
\leq \norm{Ta}{1}\norm{f}{\infty} 
\leq A\, \norm{f}{\infty}.
\end{equation} 
Now we show that $T^* f$ belongs to $\mathfrak{bmo}(M)$ and that 
$$
\norm{T^*f}{\mathfrak{bmo}} \leq 3\, A\,  \norm{f}{\infty}
\quant f \in L^{\infty}(M).
$$
Suppose that $B$ is a ball of radius at most $1$; we have
\begin{equation} \label{e:est inftyIII}
\Bigl[\int_B \mod{T^*f-(T^*f)_B}^{p'} \wrt \mu\Bigr]^{1/p'}
= \sup_{\norm{\vp}{\lp{B}} = 1}
      \Bigmod{\int_B \vp \, \bigl(T^*f-(T^*f)_B\bigr) \wrt \mu}.
\end{equation} 
Observe that 
\begin{equation} \label{e:est inftyII}
\begin{aligned}
\int_B \vp \, \bigl(T^*f-(T^*f)_B\bigr) \wrt \mu
& = \int_B \bigl(\vp-\vp_B\bigr) \, \bigl(T^*f-(T^*f)_B\bigr) \wrt \mu \\
& = \int_B \bigl(\vp-\vp_B\bigr) \, T^*f \wrt \mu.
\end{aligned}
\end{equation} 
Since $\norm{\vp}{\lp{B}} = 1$,
$$
\bigmod{\vp_B}
 \leq \Bigl[\frac{1}{\mu(B)} \, \int_B \mod{\vp}^p \wrt \mu\Bigr]^{1/p} \;
 \leq \;\mu(B)^{-1/p}.
$$
Then 
$$
\begin{aligned}
\norm{\vp-\vp_B}{\lp{B}}
& \leq \norm{\vp}{\lp{B}} + \mod{\vp_B} \, \mu(B)^{1/p} \\
& \leq 2,
\end{aligned}
$$
so that $(\vp-\vp_B)/(2\,\mu(B)^{1/p})$ is a standard $p$-atom, and 
\eqref{e:est infty} implies that 
$$
\Bigmod{\int_B (\vp-\vp_B) \,\,  T^*f \wrt \mu}
\leq 2\,A\, \norm{f}{\infty}\,\mu(B)^{1/p}.
$$
Thus, by \eqref{e:est inftyII} and \eqref{e:est inftyIII}, we may conclude that
for every ball $B$ of radius at most~$1$
$$
\Bigl[\frac{1}{\mu(B)} \, 
\int_B \mod{T^*f-(T^*f)_B}^{p'} \wrt \mu\Bigr]^{1/p'}
\leq 2\,A\, \norm{f}{\infty}.
$$
Now take a ball $B$ of radius exactly equal to $1$. We have
$$
\Bigl[\int_B \mod{T^*f}^{p'} \wrt \mu\Bigr]^{1/p'}
= \sup_{\norm{\vp}{\lp{B}} = 1}
      \Bigmod{\int_B \vp \, T^*f \wrt \mu}.
$$
The function $\vp/\mu(B)^{1/p}$ is a global $p$-atom, and, by \eqref{e:est infty}, 
$$
\Bigmod{\int_B \vp \,\,  T^*f \wrt \mu}
\leq A\, \norm{f}{\infty}\,\mu(B)^{1/p}.
$$
Therefore, for every ball $B$ of radius $1$
$$
\Bigl[\frac{1}{\mu(B)} \, 
\int_B \mod{T^*f}^{p'} \wrt \mu\Bigr]^{1/p'}
\leq A\, \norm{f}{\infty}.
$$
Combining the above estimates, we get
$$
\norm{T^*f}{\mathfrak{bmo}} \leq \norm{T^*f}{\mathfrak{bmo}^{p'}} \leq 3\, A\,  \norm{f}{\infty}
\quant f \in L^{\infty}(M),
$$
as required.

Now we prove that $T$ extends to a bounded operator from $\frh^1(M)$ to $L^1(M)$. 
Observe that $X^p$ and $\frh^{1,p}_{\textrm{fin}}(M)$ coincide as vector spaces. 
For every $g$ in $\frh^{1,p}_{\textrm{fin}}(M)$ and every $f$ in $L^{\infty}(M)$
$$
\begin{aligned}
 |\prodo{Tg}{f}|&=|\prodo{g}{T^*f}|\\
& \leq C\, \norm{g}{\frh^1} \norm{T^*f}{\mathfrak{bmo}}\\
& \leq 3\, C\, A\, \norm{g}{\frh^1} \norm{f}{\infty}.
\end{aligned}
$$
By taking the supremum of both sides over all functions $f$ in $L^{\infty}(M)$ with $\norm{f}{\infty}=1$, we obtain that
$$
\norm{T g}{1} \leq 3\, C\, A\, \norm{g}{\frh^1} 
\quant g \in \frh^{1,p}_{\textrm{fin}}(M).
$$
Since $\frh^{1,p}_{\textrm{fin}}(M)$ is dense in $\frh^1(M)$ with respect to the norm of $\frh^1(M)$, the required conclusion follows by a density argument.
\end{proof}

Suppose that $T$ is a bounded linear operator on $L^2(M)$. Then $T$ is automatically defined on $\frh^{1,2}_{\textrm{fin}}(M)$. If we assume that
$$
A:=\sup\{ \norm{T a}{1}: \hbox{\emph{$a$ is a $2$-atom}} \} 
< \infty,
$$ 
then, by the previous theorem, the restriction of $T$ to $\frh^{1,2}_{\textrm{fin}}(M)$ has a unique bounded extension to an operator $\wt T$ from 
$\frh^{1}(M)$ to $L^1(M)$. We wonder if the operators $T$ and $\wt T$ are consistent, i.e., if they agree on the intersection $\frh^1(M) \cap L^2(M)$
of their domains. As in the case of the same problem on the space $H^1(M)$ (see \cite[Prop~4.2]{MSV}), the answer is in the affirmative, as shown in the next
proposition.

\begin{proposition}
Suppose that $T$ is a bounded linear operator on $L^2(M)$ and that
$$
A:=\sup\{ \norm{T a}{1}: \hbox{\emph{$a$ is a $2$-atom}} \} 
< \infty.
$$  
Denote by $\wt T$ the unique bounded extension of the restriction of $T$ to $\frh^{1,2}_{\textrm{fin}}(M)$ to an operator from $\frh^{1}(M)$ to $L^1(M)$.
Then the operators $T$ and $\wt T$ coincide on $\frh^1(M) \cap L^2(M)$.
\end{proposition}

\begin{proof}
Assume that $f$ is in $L^2(M) \cap L^{\infty}(M)$ and that $g$ is in $L^2_{c}(M)$. 
Denote by $T^*$ the transpose operator of $T$ (as an operator on $L^2(M)$). Then 
$$
\int_{M} g\, T^*f \wrt {\mu }
= \int_{M} Tg\, f \wrt {\mu }.
$$
Since $g$ is in $\frh^{1,2}_{\textrm{fin}}(M)$ and the operators $T$ and 
$\wt T$ agree on $\frh^{1,2}_{\textrm{fin}}(M)$, we get
$$
\int_{M} Tg\, f \wrt {\mu }
= \int_{M} \wt T g\, f \wrt {\mu }.
$$
Denote by $(\wt T)^*$ the transpose of $\wt T$ as an operator from $\frh^{1}(M)$ to $L^1(M)$. 
Then
$$
\int_{M} \wt T g\, f \wrt {\mu }
= \prodo{g}{(\wt T)^*f}.
$$
Since $(\wt T)^*f$ is in $\mathfrak{bmo}(M)$ and $g$ is in $\frh^{1,2}_{\textrm{fin}}(M)$, 
we can write the last scalar product $\prodo{g}{(\wt T)^*f}$ (with respect
to the duality between $\frh^1(M)$ and $\mathfrak{bmo}(M)$) as
$$
\prodo{g}{(\wt T)^*f}=\int_{M}  g\, (\wt T)^*f \wrt {\mu }.
$$
Thus, combining the above equalities, we obtain that
$$
\int_{M}  g\, [T^*f-(\wt T)^*f] \wrt {\mu }=0
\quant g \in L^2_{c}(M),
$$
i.e., for all $g$ in $X^2$. This implies that $T^*f-(\wt T)^*f=0$ is in the dual space of $X^2$, i.e., in $L^2_{\textrm{loc}}(M)$. Thus $T^*f=(\wt T)^*f$
almost everywhere.\\
Now, suppose that $f$ is in $L^2(M) \cap L^{\infty}(M)$ and that $g$ is in $\frh^1(M)\cap L^{2}(M)$. Then
$$
\begin{aligned}
\int_{M} Tg\, f \wrt {\mu } &=\int_{M} g\, T^*f \wrt {\mu }\\
&=\int_{M}  g\, (\wt T)^*f \wrt {\mu }\\
&=\int_{M} \wt T g\, f \wrt {\mu }.
\end{aligned}
$$
Thus we have obtained that
$$
\int_{M} [Tg-\wt T g]\, f \wrt {\mu } =0
$$
for an arbitrary $f$ in $L^2(M) \cap L^{\infty}(M)$. This implies that $Tg=\wt T g$ for all $g$ in $\frh^1(M)\cap L^{2}(M)$.
\end{proof}

\section{Applications to SIO}
\label{sec:9}

The purpose of this section is to show that the Hardy space $\ghu{M}$
may be used to obtain endpoint estimates for interesting 
singular integral operators on Riemannian manifolds.  

Hereafter in this section, we assume that $M$ is a complete connected 
noncompact $n$-dimensional Riemannian manifold with \emph{bounded geometry}, 
that is with Ricci curvature bounded from below and positive injectivity radius. 
We view $M$ as a measured metric space with respect to the Riemannian distance and measure. 
Clearly the MP property holds (with $R_0=0$).
Furthermore, it is well known that manifolds with bounded geometry possess the LDP, 
as a consequence of the Bishop-Gromov comparison theorem 
(see, for instance, \cite{Gr1}, \cite[Thm III.4.5]{Ch}).  
Thus, the theory of local Hardy spaces $\ghu{M}$ developed
in the previous chapters applies to this setting.  
Denote by~$-\cL$ the Laplace--Beltrami operator on $M$: 
$\cL$ is a symmetric operator on $C_c^\infty(M)$ and its closure
is a self adjoint operator on $\ld{M}$ which we still denote by~$\cL$. 

We consider the (translated) Riesz transforms $R_a := \nabla (a\cI+\cL)^{-1/2}$, where $\nabla$
denotes the Riemannian gradient, and $a$ is a positive
number, and spectral multipliers of $\cL$ satisfying a Mihlin type condition at infinity.

The latter operators are treated in \cite{T2} and in \cite{MMV2}.
A comparison between the results obtained therein and our result
is in order. 
We extend the result in \cite{T2} by relaxing significantly the
assumptions on the geometry of $M$, as already illustrated in the Introduction.
In \cite{MMV2} the Riemannian manifold $M$
is assumed to have bounded geometry in the same sense as here,
but an additional hypothesis is made, i.e., that the bottom $b$ of the 
$L^2$ spectrum of $\cL$ is strictly positive.  This
assumption rules out, for instance, all Riemannian manifolds of
polynomial volume growth \cite{Br}.  The reason for this additional assumption 
is that the local Hardy space $H_1^1(M)$ used in \cite{MMV2}
is known to interpolate with $\ld{M}$ to give $\lp{M}$, $1<p<2$,
only when $b>0$.  

The problem of establishing endpoint estimates for $R_a$ when $p=1$
in the setting of noncompact Riemannian manifolds has been widely studied. 
In particular, T. Coulhon and X.T. Doung \cite{CD} proved
that if $M$ is locally doubling, of exponential 
growth, and supports an $L^2$-scaled Poincar\'{e} inequality, 
then $R_a$ is of weak type $1$. 
Russ \cite{Ru} complemented this result by showing that, for $a$ large enough,  
$R_a$ is bounded from the atomic Hardy space $H_1^1(M)$ to $L^1(M)$. 
Note, however, that Russ' result is known to interpolate with 
$\ld{M}$ to give $\lp{M}$ estimates only when $M$ has bounded geometry
and spectral gap (see \cite{CMM1} and the remarks above).

Here we prove, under the assumption that $M$ has bounded geometry, 
that if $a$ is suitably large, then $R_a$ is bounded from $\frh^1(M)$ to $L^1(M)$. 
This result complements the analogous result in \cite{CMM1}. 

\subsection{Spectral multipliers}
\label{sec: Hone}

First we define the class of symbols which will 
be needed in the statement of Theorem~\ref{t: lim huno}.  

\begin{definition}  \label{d: Hormander at infinity}
Suppose that $J$ is a positive integer and that $W$
is in $\BR^+$.
Denote by~$\bS_{W}$ the strip $\{ \zeta \in \BC: \Im (\zeta)
\in (-W,W)\}$ and by $H^\infty (\bS_{W};J)$ the vector space of 
all bounded \emph{even} holomorphic functions~$f$ in $\bS_{W}$ for which
there exists a positive constant $C$ such that
\begin{equation} \label{f: SsigmaJ}
\mod{D^j f(\zeta)} 
\leq {C} \, {(1+\mod{\zeta})^{-j}}
\quant \zeta\in \bS_{W} \quant j \in \{0,1,\ldots,J\}.
\end{equation}
We denote by $\norm{f}{\bS_{W};J}$ the infimum of all constants $C$ 
for which (\ref{f: SsigmaJ}) holds.
If $\norm{f}{\bS_{W};J} < \infty$ we say that $f$ satisfies a \emph{Mihlin 
condition of order $J$ at infinity on $\bS_{W}$}.
\end{definition}


Denote by $\om$ an even function in 
$C_c^\infty(\BR)$ which is supported in $[-3/4,3/4]$, is equal to~1
in $[-1/4,1/4]$, and satisfies 
$$
\sum_{j\in \BZ} \om(t-j) = 1
\quant t \in \BR.
$$
Denote by $\cD$ the operator $\sqrt{\cL-b}$, where $b$ denotes the bottom
of the $L^2$ spectrum of $\cL$.  Clearly spectral multipliers of $\cL$
may equivalently be expressed as spectral multipliers of $\cD$ 
(with a different multiplier).  
Recall that the heat semigroup is the one-parameter family $\{\cH_t\}_{t\geq 0}$ 
defined, at least on $\ld{M}$, by 
$$
\cH_t f:= \e^{-t\cL}f
\quant f \in \ld{M}.
$$ 
It is well known that $\cH_t$ extends to a contraction semigroup
on $\lp{M}$ for all $p\in [1,\infty]$.  Furthermore,
since $M$ has Ricci curvature bounded from below, 
the heat semigroup $\{\cH^t\}$ satisfies the following 
ultracontractivity estimate \cite[Section 7.5]{Gr1}
\begin{equation} \label{f: special ultrac}
\bigopnorm{\cH^t}{1;2}
\leq C\, \e^{-bt} \, t^{-n/4} \, (1+t)^{n/4-\de/2}  \quant t \in \BR^+
\end{equation}
for some $\de$ in $[0,\infty)$.
Recall also that a lower bound for the Ricci curvature implies also
an upper bound of the volume growth of $M$ (see (\ref{f: volume growth})).  
Indeed, there are positive constants $\al$, $\be$ and $C$ such that
\begin{equation} \label{f: volume growth}
\mu\bigl(B(p,r)\bigr)
\leq C \, r^{\al} \, \e^{2\be \, r}
\quant r \in [1,\infty) \quant p \in M.
\end{equation}

The following result should be compared with \cite[Proposition~B.5]{T2}.  It provides
an endpoint result to the multiplier theorem \cite[Thm~]{T1}.  

\begin{theorem} \label{t: lim huno}
Assume that $\al$ and $\be$ are as in (\ref{f: volume growth}),
and $\de$ as in (\ref{f: special ultrac}). 
Denote by $N$ the integer $[\!\![n/2+1]\!\!] +1$.  
Suppose that $J$ is an integer $>\max\, \bigl(N+2+\al/2-\de,
N + 1/2\bigr)$.
Then there exists a constant $C$ such that 
$$
\opnorm{m(\cD)}{\frh^1}
\leq C \, \norm{m}{\bS_{\be};J}
\quant m\in H^\infty\bigl(\bS_{\be};J\bigr).
$$
\end{theorem}

\begin{proof}
%
We claim that it suffices to prove that for each $2$-atom $a$ at scale~$1$,
the function $m(\cD)\, a$ may be written
as the sum of $2$-atoms supported in balls of $\cB_1$, 
with $\ell^1$ norm of the coefficients controlled
by $C \, \norm{m}{\bS_{\be};J}$.

Indeed, suppose that $f$ is a function in $\frh^1(M)$ and that
$f = \sum_j \la_j \, a_j$ is an atomic decomposition of $f$ with
$\norm{f}{\frh^1}\geq \sum_j \mod{\la_j} -\vep$. 
Since for each $2$-atom $a$ we have $\norm{m(\cD)a}{1}\leq C \, \norm{m}{\bS_{\be};J}$, by Theorem \ref{t:fin}   
$m(\cD)$ extends to a
bounded operator from $\frh^1(M)$ to $\lu{M}$.
Then $m(\cD) f = \sum_j \la_j \, m(\cD) a_j$, where the 
series is convergent in $\lu{M}$. But the partial sums of
the series $\sum_j \la_j \, m(\cD) a_j$ is a Cauchy sequence in $\frh^1(M)$,
hence the series is convergent in $\frh^1(M)$, and the sum must be
the function $m(\cD)f$.
Therefore 
$$
\begin{aligned}
\norm{m(\cD)f}{\frh^1}
& \leq \sum_j \mod{\la_j} \, \norm{m(\cD)a_j}{\frh^1} \\
& \leq C\, \, \norm{m}{\bS_{\be};J}\,  \sum_j \mod{\la_j} \\
& \leq C\, \, \norm{m}{\bS_{\be};J}\, (\norm{f}{\frh^1}+ \vep),
\end{aligned}
$$
and the required conclusion follows by taking
the infimum of both sides with respect to all admissible
decompositions of $f$.  

It has already been shown in the proof of \cite[Thm~3.4]{MMV2} 
that the claim holds for standard atoms. Therefore it suffices to prove it 
for global atoms.   

As in the proof of \cite[Thm~3.4]{MMV2}, 
we split the operator $m(\cD)$ into the sum of two operators 
and analyse them separately.
The functions $\wh\om\ast m$ and $m-\wh\om\ast m$ 
($\om$ is the cut-off function defined above) are bounded. 
Define the operators~$\cS $ and $\cT $ spectrally by 
$$
\cS 
= (\wh\om\ast m) (\cD)
\qquad\hbox{and}\qquad
\cT 
= (m-\wh\om\ast m) (\cD).
$$
Thus
$
m(\cD) 
= \cS  + \cT .
$

\medskip\noindent
Suppose that $a$ is a global $2$-atom
supported in $B_1(p)$ for some $p$ in~$M$. 
Observe that the function $\wh{\om}\ast m$ is bounded and 
\begin{equation}  \label{f: norm est whomastm}
\norm{\wh{\om}\ast m}{\infty} \leq C\, \norm{m}{\infty}
\leq C \, \norm{m}{\bS_{\be};J}. 
\end{equation}
Therefore, $(\wh\om\ast m) (\cD)$ is bounded on $\ld{M}$ by the spectral theorem, and
$$
\opnorm{(\wh\om\ast m) (\cD)}{2}\
\leq \norm{\wh\om\ast m}{\infty}
\leq C \, \norm{m}{\bS_{\be};J}.  
$$
We have used (\ref{f: norm est whomastm}) in the second inequality above.
Observe that the support of the kernel of the operator $(\wh\om\ast m) (\cD)$ 
is contained in $\{(x,y): d(x,y) \leq 1\}$, for $\cL$ possesses the finite
propagation speed property,
hence the function $(\wh\om\ast m)(\cD)a$ is supported in the 
ball with centre~$p$ and radius $2$. 
Moreover, 
$$
\begin{aligned}
\norm{(\wh\om\ast m) (\cD)a}{2}
&\leq C\, \opnorm{(\wh\om\ast m) (\cD)}{2}\, \norm{a}{2} \\
&\leq C\, \norm{m}{\bS_{\be};J}\,  \mu(B_1(p))^{-1/2} \\
&\leq C\, \norm{m}{\bS_{\be};J}\,  \mu(B_2(p))^{-1/2}. \\
\end{aligned}  
$$
We have used the LDP in the last inequality.
Thus, $(\wh\om\ast m) (\cD)a$ is a constant 
multiple of a global atom at scale $2$ and, by 
Lemma~\ref{l: economical decomposition}, 
$$
\norm{(\wh\om\ast m) (\cD) a}{\frh^1}  \leq C \,  \norm{m}{\bS_{\be};J}.
$$

Now we analyse $\cT a$. In the proof of \cite[Thm~3.4]{MMV2}  
it is shown that if $b$ is a standard $2$-atom, then $\cT b$
may be decomposed as 
$$
\cT b 
= \sum_{j=1}^\infty \la_j \, b_j,
$$
where $b_j$ is an atom at scale $j+2$, and 
\begin{equation} \label{f: est laj}
\mod{\la_j} 
\leq C\, \norm{m}{\bS_{\be};J} \, j^{N+\al/2-J-\delta}
\quant j =1,2,3,\ldots
\end{equation}
A close examination of the proof reveals that 
the cancellation property of $b$ is used to show that the 
atoms $b_j$ also have this property, but it is not required in
the proof of (\ref{f: est laj}).
Thus, by arguing as in Step~IV in the proof of \cite[Thm~3.4]{MMV2},
we may conclude that $\cT a$ may be written as 
$$
\cT a 
= \sum_{j=1}^\infty \la_j \, a_j,
$$  
where $a_j$ is a global $2$-atom at scale $j+2$ and $\la_j$ satisfies
estimate (\ref{f: est laj}).
Now  Lemma~\ref{l: economical decomposition} 
(and its version for Riemannian manifolds \cite[Lemma~5.7]{MMV2}) 
imply that there exists a constant $C$ such that 
$$
\norm{a_j}{\frh^1} 
\leq C\, j  
\qquad j=1,2,3,\ldots
$$ 
Therefore
$$
\norm{\cT a}{\frh^1}
\leq C \, \norm{m}{\bS_{\be};J}\sum_{j=1}^\infty j^{1+N+\al/2-J-\delta} 
\leq C\, \norm{m}{\bS_{\be};J},
$$
where $C$ is independent of $a$.  Hence $\cT$ extends to a bounded
operator from $\ghu{M}$ to $\ghu{M}$.

So far, we have proved that there exists
a constant $C$ such that for every  global atom~$a$ 
$$
\norm{\cS a}{\frh^1} + \norm{\cT a}{\frh^1}
\leq  C\, \norm{m}{\bS_{\be};J}. 
$$ 
Hence
$$
\norm{m(\cD)a}{\frh^1}
\leq  C\, \norm{m}{\bS_{\be};J}. 
$$ 
The required conclusion follows from the claim at the beginning of the proof.
\end{proof}
%

\subsection{The translated Riesz transform}

We shall need the following local estimate for 
the space derivative of the heat kernel.

\begin{lemma}\label{l:heat4}
There exists $\eta>0$ such that for all $y \in M$, $t>0$
$$
\int_{d(x,y)\geq \sqrt{t}} |\nabla_x h_s(x,y)| \wrt\mu(x)  
\leq \left\{\begin{array}{ll}
C e^{-\eta t/s} s^{-1/2} & \quant  s \in (0,1]\\
\\
C e^{-\eta t/s} e^{cs} s^{-1/2} & \quant  s\in (1,\infty) \\
\end{array}. \right.\\
$$
\end{lemma}

This result is stated in \cite{CD}, though  
its proof is given in full detail only in the case where $M$
is globally doubling.  However, it is not hard to modify
the argument to produce a proof of Lemma~\ref{l:heat4}.
The proof hinges on upper estimates for the heat kernel 
and its time derivatives 
(see \cite{Gr1,Gr2,D}) and on 
weighted estimates for the space derivative of the heat kernel (\cite{CD}).

\begin{theorem}\label{t:riesz}
There exists $a>0$ such that the translated Riesz transform $\nabla(a \cI+ \cL)^{-1/2}$ is bounded from $\frh^1(M)$ to $\lu{M}$.
\end{theorem}

\begin{proof}
We know that if $a$ is large enough, 
then $\nabla(a \cI+ \cL)^{-1/2}$ is bounded from 
$H^1(M)$ to $\lu{M}$ by \cite{Ru}. 
Therefore it suffices to show that the kernel $k$ of 
$\nabla(a \cI+ \cL)^{-1/2}$ satisfies the condition
\begin{equation}\label{e:ker}
\sup_{y \in M}\int_{B_2(y)^{c}}|k(x,y)|\ \wrt\mu(x) < \infty
\end{equation}
and then apply \cite[Prop~4.5]{CMM3}.
The kernel is given, off the diagonal,~by
$$
k(x,y)=\int_{0}^{+\infty}\frac{\e^{-as}}{\sqrt{s}}\nabla_{x}h_{s}(x,y)\ \wrt s.
$$
By Fubini's theorem, we obtain
$$
\begin{aligned}
\int_{B_2(y)^{c}}|k(x,y)|\ \wrt\mu(x)
& =\int_{B_(y)^{c}} \left| \int_{0}^{+\infty}\frac{\e^{-as}}{\sqrt{s}}\nabla_{x}h_{s}(x,y)\ \wrt s\right| \wrt\mu(x)\\
& \leq \int_{0}^{+\infty}\frac{\e^{-as}}{\sqrt{s}} \int_{d(x,y)\geq 2}|\nabla_{x}h_{s}(x,y)|\wrt\mu(x) \wrt s\\
& := I_1+I_2,
\end{aligned}
$$
where 
$$
I_1=\int_{0}^{1}\frac{\e^{-as}}{\sqrt{s}} \int_{d(x,y)\geq 2}|\nabla_{x}h_{s}(x,y)|\wrt\mu(x) \wrt s
$$
and 
$$
I_2=\int_{1}^{+\infty}\frac{\e^{-as}}{\sqrt{s}} \int_{d(x,y)\geq 2}|\nabla_{x}h_{s}(x,y)|\wrt\mu(x) \wrt s.
$$
Now we apply Lemma~\ref{l:heat4} to estimate the inner integrals. We get
$$
I_1 \leq C \int_{0}^{1}\frac{\e^{-as-4{\eta}/{s}}}{s} \wrt s \leq C \int_{0}^{1}\frac{\e^{-4{\eta}/{s}}}{s^2} \wrt s
=C \, \frac{\e^{-4\eta}}{4\eta},
$$ 
and
$$
I_2 
\leq C\int_{1}^{+\infty} \frac{\e^{-(a-c)s-4{\eta}/{s}}}{s} \wrt s 
\leq C\int_{1}^{+\infty} \e^{-(a-c)s}\wrt s.
$$
Note that the last integral converges only when $a>c$.
Therefore (\ref{e:ker}) holds if $a>c$ and for such $a$ 
the operator $\nabla(a \cI+ \cL)^{-1/2}$ extends to a 
bounded operator from $\frh^1(M)$ to $\lu{M}$, as required.
\end{proof}

\end{document}